\documentclass{hha-nostro}
\usepackage{amsmath,amsfonts,amssymb,amsrefs,amscd,latexsym,euscript,xypic, verbatim}

\usepackage{etex}

\usepackage{amstext}
\usepackage{graphicx}
\usepackage[active]{srcltx}
\usepackage[all]{xy}
\usepackage{graphics}
\usepackage[utf8]{inputenc}
\usepackage{mathtools}
\usepackage{dsfont}
\usepackage{multicol}
\usepackage[shortlabels]{enumitem}
\usepackage{multirow}

\usepackage{blindtext}

\usepackage{diagbox}
\usepackage{microtype}
\usepackage{setspace}

\usepackage{tikz} 
\usepackage{tikz-cd}
\usetikzlibrary{ 
	calc,%
	arrows,%
	shapes
} 

\usetikzlibrary{matrix}
\usetikzlibrary{positioning}

\usepackage{enumitem}
\newenvironment{axioms}
{\enumerate[label=\textbf{MC\arabic*.}, ref=MC\arabic*]}
{\endenumerate}
\makeatletter
\newcommand\varitem[1]{\item[\textbf{MC\arabic{enumi}\rlap{$#1$}.}]%
	\edef\@currentlabel{MC\arabic{enumi}{$#1$}}}
\makeatother

\tikzcdset{arrow style=tikz, diagrams={>=stealth}}

\usepackage{thmtools}

\usepackage{theoremref}

\swapnumbers

\theoremstyle{definition}
\newtheorem{point}[theorem]{}

\numberwithin{equation}{section}

\usepackage[hidelinks]{hyperref}
\usepackage{amsthm}
\usepackage[capitalize, nameinlink]{cleveref}

\DeclareMathOperator{\Img}{im}

\DeclareMathOperator{\coker}{coker}
\DeclareMathOperator{\colim}{colim}

\newcommand{\M}{\mathcal{M}} 
\newcommand{\Ch}{\text{\rm Ch}} 
\newcommand{\ch}{\text{\rm ch}} 
\newcommand{\D}[1][n]{D^{#1}}
\newcommand{\Sp}[1][n]{S^{#1}}
\newcommand{\id}{1}

\newcommand{\intS}[2][n]{I^{#1}\left[#2\right)}

\newcommand{\vect}{\text{\rm vect}_{K}}

\newcommand{\posr}{\left[0,\infty\right)}

\newcommand{\tame}{\text{\rm tame}\left(\posr, \ch\right)}
\newcommand{\tamem}{\text{\rm tame}\left(\posr,\mathcal{M}\right)}

\xyoption{2cell}

\newcommand{\xrightarrowdbl}[2][]{%
  \xrightarrow[#1]{#2}\mathrel{\mkern-14mu}\rightarrow
}

\begin{document}
	
\title{Invariants for tame parametrised chain complexes}


\author{Wojciech Chach{\'o}lski}             
\email{wojtek@math.kth.se}
\address{KTH Stockholm}

\author{Barbara Giunti}             
\email{barbara.giunti@unimore.it}
\address{University of Modena and Reggio Emilia}

\author{Claudia Landi}             
\email{claudia.landi@unimore.it}
\address{University of Modena and Reggio Emilia}

\classification{55PXX, 55N31.}

\keywords{Topological data analysis, cofibrant approximations, minimality, persistence theory.}

\begin{abstract}
	We set the foundations for a new approach to Topological Data Analysis (TDA) based on homotopical methods at chain complexes level.
	We present the category of tame parametrised chain complexes as a comprehensive environment that includes several cases that usually TDA handles separately, such as persistence modules, zigzag modules, and commutative ladders.
	We extract new invariants in this category using a model structure and various minimal cofibrant approximations.
	Such approximations and their invariants retain some of the topological, and not just homological, aspects of the objects they approximate. 
\end{abstract}

\maketitle

\section*{Introduction}\label{intro}	

Data analysis is often about simplifying, ignoring most of the information available and extracting what might be meaningful for the task at hand. 
The same strategy of extracting summaries is also at the core of topology. 
In recent years, these two branches have merged, giving rise to Topological Data Analysis (TDA) \cite{carlsson_topology_data}.  

TDA can benefit from a broad spectrum of existing homotopical tools for extracting such summaries. 
Currently, the most popular is persistent homology.  
The first step in persistence theory is to transform data into spatial information via, for example, the Vietoris-Rips construction.
The second step is typically the extraction of the homology of the obtained spaces, resulting in a so-called persistence module, effectively studied by enumerating its indecomposables \cite{oudot}.

Despite its success, TDA based on persistent homology has some limitations. 
Firstly, one is limited to  objects  for which it is possible to list their  indecomposable summands \cite{oudot},
and whose decompositions can be computed algorithmically.
For example, the class of \emph{commutative ladders} cannot be analysed using its indecomposables because it is of wild representation type \cite{escolar_ladders_ind, commutative_ladders}.
On the other hand, the indecomposables of the class of \emph{zigzag modules} are fully described, but so far there is no efficient software to analyse them \cite{zigzag_carlsson, zigzag_3}.
Secondly, applying homology might be too drastic, disregarding a large amount of geometric information.

The main goal of this paper is to show how to use homotopy theory not only to overcome the previous issues but also to open the way towards new invariants. 
The same strategy of disregarding some information and focusing on aspects that might be relevant is also at the core of homotopy theory, with {\em colocalization} being an example of such a process. 
In colocalization, the simplification is achieved by approximating arbitrary objects by other objects that are simpler and more manageable, such as the class of \emph{cofibrant objects} in a model category. 
Our work is based on the realisation that the category $\tame$ of tame $\posr$-parametrised chain complexes over a field admits a model structure for which there is a surprisingly simple decomposition theorem describing all indecomposable cofibrant objects (see Theorem~\ref{dec_theorem}). 
This is so even though the entire category $\tame$ is of wild representation type. 
The structure theorem \ref{dec_theorem} identifies cofibrant objects in $\tame$ with sequences of persistence diagrams augmented with points on the diagonal (see Section~\ref{section_Betti}) which we call Betti diagrams.

The cofibrant objects can be then used to approximate arbitrary objects in $\tame$. 
Proving that such minimal approximations exist (see Theorem~\ref{sdvgsdfhsfgn}) has been essential in this work.

Model categories are convenient for ensuring the existence of certain morphisms or approximations. 
A common difficulty in working with model categories, however, is the lack of algorithmic constructions producing such morphisms and approximations. 
Approximations in model categories are often constructed using universal properties and require performing large limits.  Extracting calculable invariants from such approximations, which is essential in TDA, is often not feasible. 
In this article, we make a great effort to describe all the constructions, factorisations, and approximations explicitly. All the steps we perform for the tame $\posr$-parametrised chain complexes in perspective can be implemented.

Considering tame $\posr$-parametrised chain complexes instead of vector spaces has another advantage.
Persistence modules, zigzag modules, and commutative ladders are special objects in $\tame$. 
Thus, this category allows for a comprehensive theory in which  different objects 
that are handled separately by standard persistence theory can be studied and compared together. 
Furthermore, for persistence modules, cofibrant minimal approximations provide complete invariants (see \ref{asdfdfhdfn}).

In conclusion, we propose a refined approach to the persistence pipeline:
first, convert the input into a parametrised simplicial complex.
Second, extract a parametrised chain complex. 
Third, form a minimal cofibrant approximation of the extracted parametrised chain complex.
Finally, represent the minimal cofibrant approximation by its Betti diagrams.
\smallskip

{\bf Related works.}
The model structure described in Section~\ref{section_par_objects} is a special case of a projective model
structures on a tame subcategory of functor categories \cite{heller, hirschhorn}.

The structure theorem~\ref{dec_theorem} describing cofibrant objects in $\tame$
appears also in, for example, \cite{framed_morse, structural_filtered, related_work_decomposition},
although the language of model categories is not used there.
An interpretation from the point of view of Morse theory was given in \cite{framed_morse}.
In \cite{structural_filtered}, Meehan, Pavlichenko and Segert show that the category of filtered chain complexes is a Krull-Schmidt category.
In \cite{related_work_decomposition}, Usher and Zang generalise the theory of barcodes to filtered Floer-type complexes, considering chain complexes of infinite dimension whose parametrisation is not tame.
They prove a singular value decomposition theorem for such complexes
and identify two types of barcodes of them: the \emph{verbose} and the \emph{concise}. 
In the finite case, such barcodes correspond respectively to the Betti diagrams and the minimal Betti diagrams of cofibrant objects in our setting (see Section~\ref{section_Betti}).

The point of view of homotopy theory is entering the TDA subject also for purposes different from ours. 
For example, \cite{Blumberg17, persistent_homotopy, frosini1999,jardine2020} are about lifting the stability theorem of persistence to homotopy stability theorems, to make it applicable to a wider class of datasets.

\section{Minimality}\label{section_minimality}

Let $\mathcal M$ be a model category \cite{dwyerspalinski, quillen}.
This means that three classes of morphisms in $\mathcal{M}$ are chosen:
{\bf weak equivalences} ($\xrightarrow{\sim}$), {\bf fibrations} ($\twoheadrightarrow$), and {\bf cofibrations}
($\hookrightarrow$).
These classes and $\M$ are required to satisfy  the following axioms:
\begin{axioms}
	\item \label{MC1} Finite limits and colimits exist in $\mathcal{M}$.
	\item \label{MC2} If $f$ and $g$ are morphisms in $\mathcal{M}$ such that $g f$ is defined and if two of the three morphisms $f$, $g$, $gf$ are weak equivalences, then so is the third.
	\item \label{MC3} If $f$ is a retract of $g$ and $g$ is a fibration, a cofibration, or a weak equivalence, then so is $f$.
	\item \label{MC4} Consider a  commutative square 
	in $\mathcal M$ consisting of the solid morphisms:
	\[
	\begin{tikzcd}[row sep=13pt]
	X \arrow[hook']{d}[swap]{\alpha} \arrow[r] 
	& E\arrow[two heads]{d}{\beta}
	\\
	Y\arrow[r] \arrow[ru,dotted]
	& B
	\end{tikzcd}
	\]
	Then a morphism, depicted by the dotted arrow and making this  diagram commutative, exists under either of the following two assumptions:  (i) $\alpha$ is a cofibration and a weak equivalence and $\beta$  is a fibration, or
	(ii) $\alpha$ is a cofibration and $\beta$ is a fibration and a weak equivalence.
	\item \label{MC5} Any morphism $g$ can be factored in two ways: 
	(i) $g = \beta\alpha$, where $\alpha$ is a cofibration and $\beta$ is both a fibration and a weak equivalence, and 
	(ii) $g = \beta\alpha$, where $\alpha$ is both a cofibration and a weak equivalence and $\beta$ is a fibration.
\end{axioms}

In particular, \ref{MC1} guarantees the existence of the initial object, denoted by $\emptyset$, and of the terminal object, denoted by $\ast$.
An object $X$ in a model category $\mathcal M$ is called {\bf cofibrant} if the morphism $\emptyset \to X$ is a cofibration. 
If the morphism $X\to \ast$ is a fibration, then $X$ is called {\bf fibrant}.

Axiom \ref{MC5} above guarantees existence of certain factorisations of morphisms. 
It does not specify any uniqueness. 
Typically, a morphism in a model category admits many such factorisations.
There are however model categories in which among all these  factorisations there is a canonical one called minimal \cite{auslander_reiten_smalo, Roig93}: 

\begin{definition}\label{def_minimal_morphism}
	Let $g\colon X\to Y$ be a morphism in $\mathcal M$. 
	A factorisation $g=\beta \alpha$, where $\alpha$ is cofibration and $\beta$ is a fibration and a weak equivalence, is called {\bf minimal} if every morphism $\phi$ which makes the following diagram commutative is an isomorphism:
	\[\begin{tikzcd}[row sep=5]
	& A\ar[bend left=15,two heads]{dr}{\beta}[swap,pos=0.6]{\sim} \\
	X\ar[bend left=15,hook]{ur}{\alpha}\ar[bend right=15,hook]{dr}[swap]{\alpha} \ar[bend left=10]{rr}[pos=.2,description]{g}& & Y\\
	& A\ar[bend right=15,two heads]{ur}[swap]{\beta}[pos=0.6]{\sim}\ar[bend left=15, crossing over,from=uu,pos=.6,"\phi"]
	\end{tikzcd}
	\]
	%
	
	A minimal factorisation of $\emptyset\to X$ is called a {\bf minimal cover} of $X$.
\end{definition}

According to the above definition, we can think about a minimal cover of $X$ as a morphism $\beta\colon \text{cov}(X)\to X$ such that: 
(i) $\text{cov}(X)$ is cofibrant, 
(ii) $\beta$ is both a fibration and a weak equivalence, and 
(iii) any morphism $\phi$ which makes the following diagram commutative is an isomorphism:
\[
\begin{tikzcd}[row sep=13pt]
& \text{cov}(X)\arrow[two heads]{d}[right]{\beta}[left]{\sim}
\\
\text{cov}(X)\arrow[two heads]{r}[below]{\beta}[above]{\sim} 
\arrow[bend left=10]{ur}[above]{\phi} 
& X
\end{tikzcd}
\]

Minimal factorisations are unique:

\begin{proposition}\label{min_fact_unique}
	Let $g\colon X\to Y$ be a morphism in  $\mathcal M$.
	Assume $\beta\alpha=g=\beta'\alpha'$ are minimal factorisations. Then there is an isomorphism $\phi$ making the following
	diagram commutative:
	\[
	\begin{tikzcd}[column sep=3em,row sep=13pt]
	X\arrow[hook]{r}[above]{\alpha'}\arrow[hook']{d}[left]{\alpha} 
	& A'\arrow[two heads]{d}[right]{\beta'}[left]{\sim}
	\\
	A\arrow[two heads]{r}[below]{\beta}[above]{\sim}\arrow{ur}[description]{\phi} 
	& Y
	\end{tikzcd}
	\]
\end{proposition}

\begin{proof}
	Let $\phi$ and $\psi$ be any  morphisms making the following diagram commute, which exist
	by the lifting axiom \ref{MC4}:	
	\[
	\begin{tikzcd}[column sep=4em,row sep=23pt]
	X\arrow[hook]{r}[above]{\alpha'}\arrow[hook']{d}[left]{\alpha} 
	& A' \arrow[two heads]{d}[right]{\beta'}[left]{\sim} 
	\arrow[dotted, bend left=15]{dl}[description]{\psi}
	\\
	A \arrow[two heads]{r}[below]{\beta}[above]{\sim}
	\arrow[bend left=15, dotted]{ur}[description]{\phi} 
	& Y
	\end{tikzcd}
	\]
	Then by the definition of  minimal factorisations, the compositions $\phi\psi$ and $\psi\phi$ are isomorphisms. Consequently, so are $\phi$ and $\psi$.
\end{proof}

Two objects $X$ and $Y$ in  $\mathcal M$ are called {\bf weakly equivalent} if there is a sequence of weak equivalences of the form:
\[
\begin{tikzcd}
X 
&  A_0\arrow{l}[above]{\sim}\arrow{r}[above]{\sim} 
&A_1 &\arrow{l}[above]{\sim}\cdots 
&
A_k \arrow{l}[above]{\sim}\arrow{r}[above]{\sim} 
& Y
\end{tikzcd}
\]

Similarly to factorisations of morphisms in a model category, the collection of objects weakly equivalent to a given object is large. 
There are model categories, however, where this collection contains a canonical object called a minimal representative:

\begin{definition}\label{def_minimal_object}
	An object $X$ in $\mathcal M$ is called {\bf minimal} if it is cofibrant, fibrant, and any weak equivalence $\phi\colon X\to X$ is an isomorphism.
	A  {\bf minimal representative} of an object $X$ in $\M$  is a minimal object in $\M$ which is weakly equivalent to $X$. 
\end{definition}

Minimal representatives are unique up to isomorphisms:

\begin{proposition}\label{min_rep_unique}
	Let $X'$ and $Y'$ be minimal representatives of respectively $X$ and $Y$.
	Then $X$ and $Y$ are weakly equivalent if and only if $X'$ and $Y'$ are isomorphic. 
\end{proposition}

\begin{proof}
	If $X'$ and $Y'$ are isomorphic, then $X$ and $Y$ are weakly equivalent. 
	Assume $X$ and $Y$ are weakly equivalent. 
	Then $X'$ and $Y'$ are also weakly equivalent.
	Since they are both cofibrant and fibrant there are weak equivalences $\phi\colon X'\xrightarrow{\sim} Y'$ and $\psi\colon Y'\xrightarrow{\sim} X'$. 
	By the definition of the minimality, the compositions $\phi\psi$ and $\psi\phi$ are isomorphisms. 
	Consequently, so are $\phi$ and $\psi$ and hence 
	$X'$ and $Y'$ are isomorphic.
\end{proof}

Proposition~\ref{min_fact_unique} and Proposition~\ref{min_rep_unique} ensure the uniqueness of minimal factorisations, minimal covers and minimal representatives. 
These propositions however do not imply their existence, which has to be proven separately and it does depend on the considered model category.

\begin{definition}\label{minal_axioms}
	A model category satisfies the {\bf minimal factorisation axiom} if all minimal factorisations exist in this category. 
	It satisfies the {\bf minimal representative  axiom} if all minimal representatives exist in this category.
\end{definition}

Many model categories, particularly of combinatorial flavour,  satisfy the minimal factorisation axiom. 
However the standard model structure on topological spaces~\cite{dwyerspalinski} does not. 

\section{Tame $\posr$-parametrised objects}\label{section_par_objects}

Let $\M$ be a category.
The symbol $\posr$ denotes the poset of non-negative real numbers. 
Functors of the form $X\colon \posr\to \M$ are also referred to as  $\posr$-parametrised objects.
The value of $X$ at $t$ in $\posr$ is denoted by $X^{t}$ and $X^{s\leq t}\colon X^s\to X^t$ denotes the morphism in $\M$ that $X$ assigns to $s\leq t$. 
The morphism $X^{s\leq t}$ is also referred to as the transition morphism in $X$ from $s$ to $t$.

\begin{definition}\label{def_tame}
	A sequence $\tau_0<\cdots<\tau_k$ in $ \posr$ {\bf discretises} $X\colon \posr\to \M$ if $X^{s\leq t}\colon X^s\to X^t$ may fail to be an isomorphism only when there is $a$ in $[k]$ such that $s < \tau_{a}\leq t$.
	A functor $X\colon \posr\to \M$ is called {\bf tame} if there is a sequence that discretises it.
	The symbol $\tamem$ denotes the category whose objects are tame functors $X\colon \posr\to \M$ and whose morphisms are the natural transformations.
\end{definition}

If $\tau_0<\cdots<\tau_k$ discretises $X\colon \posr\to \M$, then the transitions of the restrictions of $X$ to the intervals $[0,\tau_0)$,\ldots, $[\tau_{k-1},\tau_k)$, and $[\tau_k,\infty)$  are isomorphisms.  
Note that if $\tau_0<\cdots<\tau_k$ discretises $X\colon \posr\to \M$, then so does any of its refinements (a sequence $\mu_0<\cdots<\mu_n$ is a {\bf refinement} of $\tau_0<\cdots<\tau_k$ if $\{\tau_0\ldots,\tau_k\}$ is a subset of  $\{\mu_0\ldots,\mu_n\}$). 

\begin{point}\label{point_kan_extension}
	{\em Kan extensions.}
	Consider a sequence of $k$ composable morphisms in $\M$:
	\[
	\begin{tikzcd}[column sep=40pt]
	X^0 \ar[r, "X^{0<1}"]
	& \cdots \ar[r, "X^{{k-1}<k}"] 
	& X^{k}
	\end{tikzcd}
	\]
	
	\noindent
	Such a sequence encodes a functor $X\colon [k]\to \M$, where $[k]$ is the standard poset on the set $\{0,\ldots,k\}$.
	Consider also a sequence  $\tau_0<\cdots<\tau_k$ of elements  in $\posr$, which  encodes an inclusion of categories $\tau\colon [k]\subset \posr$.
	The {\bf (left) Kan extension} of $X$ along $\tau$ \cite[Sect. X.3]{maclane} is a functor $LX\colon \posr\to \M$ whose values are given by:
	\[
	LX^t=
	\begin{cases}
	\emptyset 
	& \text{ if } t<\tau_0
	\\
	X^{\max\{a\ |\ \tau_a\leq t\}}
	&\text{ if } t\geq \tau_0
	\end{cases}
	\]
	
	For morphisms, $LX^{s<t}$ is the identity if $\max\{a\ |\ \tau_a\leq s\}=\max\{a\ |\ \tau_a\leq t\}$, and otherwise it is the composition of:
	\[
	X^{\max\{a\ |\ \tau_a\leq s\}<\max\{a\ |\ \tau_a\leq s\}+1}, 
	\ldots, 
	X^{\max\{a\ |\ \tau_a\leq t\}-1<\max\{a\ |\ \tau_a\leq t\}}
	\]
	
	The functor $LX\colon \posr\to \M$ is tame with  $\tau_0<\cdots<\tau_k$ as a discretising sequence.   
	
	To describe a natural transformation $g\colon LX\to Y$ to any other functor $Y\colon\posr\to\M$, it is enough to specify a sequence of morphisms $\{g^{\tau_a}\colon X^{\tau_a}\to Y^{\tau_a}\}_{a=0,\dots,k}$ for which the following diagram commutes for every $a=1,\ldots,k$:
	\[
	\begin{tikzcd}[column sep=40pt,row sep=13]
	X^{\tau_{a-1}} \ar[r, "X^{\tau_{a-1}<\tau_a}"] 
	\ar[d, "g^{\tau_{a-1}}"'] 
	& X^{\tau_{a}} \ar{d}{g^{\tau_a}}
	\\
	Y^{\tau_{a-1}} \ar[r, "Y^{\tau_{a-1}<\tau_a}"] 
	& Y^{\tau_{a}}
	\end{tikzcd}
	\]
	
	If $k=0$, for an object $X$ in $\M$ (representing a functor $X\colon [0]\to\M$) and an element $\tau_0$ in $\posr$ (representing an inclusion $\tau_0\colon[0]\subset\posr$), the induced Kan extension is a functor $LX\colon\posr\to \M$ such that $LX^t=\emptyset$ if $t<\tau_0$ and $LX^t=X$ if $\tau_0\leq t$.
	In this case, the set of natural transformations $LX\to Y$ is in bijection with the set of morphisms in $\M$ from $X$ to $Y^{\tau_0}$. 
	If $k=1$, for a morphism $X^{0<1}\colon X^0\to X^1$ (representing a functor $X\colon [1]\to\M$) and two elements $\tau_0<\tau_1$ in $\posr$ (representing an inclusion $[1]\subset\posr$), the induced  Kan extension is a functor $LX\colon\posr\to \M$ such that $LX^t=\emptyset$ if $t<\tau_0$, $LX^t=X^0$ if $\tau_0\leq t<\tau_1$, and $LX^t=X^1$ if $\tau_1\leq t$. 
	In this case, the set of natural transformations $LX\to Y$ is in bijection with commutative squares of the form:
	\[
	\begin{tikzcd}[column sep=35pt,row sep=13]
	X^{0}\ar{r}{X^{0<1}}\ar{d} 
	& X^{1}\ar{d}
	\\
	Y^{\tau_0}\ar{r}{Y^{\tau_0<\tau_1}} 
	& Y^{\tau_1}
	\end{tikzcd}
	\]
	
	Let $X\colon\posr\to\M$ be tame with $0=\tau_0<\cdots<\tau_k$ as a discretising sequence.
	Then $X$ is isomorphic to the Kan extension along $0= \tau_0<\cdots<\tau_k$ of the following sequence of morphisms :
	\[
	\begin{tikzcd}[column sep=4em]
	X^0 \ar[r, "X^{0\leq \tau_1}"]
	& \cdots \ar[r, "X^{\tau_{k-1}<\tau_k}"] 
	& X^{\tau_k}
	\end{tikzcd}
	\]	
\end{point}

\begin{point}{\em Factorisation.}\label{point_factorisation}
	Let $\M$ admit all finite colimits.
	Let $g\colon X\to Y$ be a morphism in $\tamem$ and $0=\tau_0<\cdots<\tau_k$ a sequence discretising both $X$ and $Y$. 
	By induction on $a$ in $[k]$, define morphisms $\bar{g}^{\tau_a}\colon X^{\tau_a}\to Q^{\tau_a}$ and $\hat{g}^{\tau_a}\colon Q^{\tau_a}\to Y^{\tau_a}$ in $\M$ as follows:
	
	\noindent
	For $a=0:\quad (\bar{g}^{0}\colon X^{0}\to Q^{0}) :=
	(\id\colon X^{0}\to X^{0}) 
	\qquad
	(\hat{g}^{0}\colon Q^{0}\to Y^{0}):=(g^0\colon X^{0}\to Y^{0})$
	
	\noindent	
	For $a=1\ldots,k:\qquad \
	Q^{\tau_a}:=\colim (Y^{\tau_{a-1}}\xleftarrow{g^{\tau_{a-1}}} X^{\tau_{a-1}}\xrightarrow{X^{\tau_{a-1}<\tau_a}}X^{\tau_{a}})
	$
	\smallskip
	
	\noindent
	$\bar{g}^{\tau_a}\colon X^{\tau_a}\to Q^{\tau_a}$ and $\hat{g}^{\tau_a}\colon Q^{\tau_a}\to X^{\tau_a}$ are the unique morphisms making the following diagram commutative, where the inside square is pushout:
	\[
	\begin{tikzcd}[column sep=1.2cm,row sep=18pt]
	X^{\tau_{a-1}} \ar[r,"X^{\tau_{a-1}<\tau_a}"] \ar[d, "g^{\tau_{a-1}}"']
	& X^{\tau_{a}} \ar[ddr,"g^{\tau_a}", bend left=17]
	\ar[d, "\bar{g}^{\tau_a}"']
	& [-0.4cm]
	\\
	Y^{\tau_{a-1}} \ar[rrd, bend right=15, "Y^{\tau_{a-1}<\tau_a}"'] \ar[r]
	& Q^{\tau_a} \arrow{dr}[pos=0.2]{\hat{g}^{\tau_a}}
	\\
	[-0.7cm]
	& 
	& Y^{\tau_{a}}
	\end{tikzcd}
	\]	
	For $a=1,\ldots,k$, define $Q^{\tau_{a-1}<\tau_a}\colon Q^{\tau_{a-1}}\to Q^{\tau_{a}}$ to be the composition of the morphism $ Y^{\tau_{a-1}}\to Q^{\tau_{a}}$ represented by the bottom horizontal arrow in the above diagram and $\hat{g}^{\tau_{a-1}}\colon Q^{\tau_{a-1}} \to Y^{\tau_{a-1}}$. 
	Let $Q\colon\posr\to \M$ be the tame functor given by the Kan extension of the sequence of morphisms $\{Q^{\tau_{a-1}<\tau_a}\}$ along $0=\tau_0<\cdots<\tau_k$ (see \ref{point_kan_extension}). 
	Finally, denote by $\bar{g}\colon X\to Q$ and $\hat{g}\colon Q\to Y$ the natural transformation given by $\{\bar{g}^{\tau_a}\}_{a=0,\ldots,k}$ and $\{\hat{g}^{\tau_a}\}_{a=0,\ldots,k}$ (see \ref{point_kan_extension}). 
	Note that $g=\hat{g}\bar{g}$. 
	
	The isomorphism type of the functor $Q$ and the factorisation $g=\hat{g}\bar{g}$ do not depend on the choice of the sequence that discretises $X$ and $Y$. 
	If $\bar{f}\colon X\to P$ and $\hat{f}\colon P\to Y$ are natural transformations constructed with respect to another such a sequence, then there is a unique isomorphism $\phi\colon Q\to P$ for which the following diagram commutes:
	\[\begin{tikzcd}[row sep=6]
	& Q\ar[bend left=15]{dr}{\hat{g}} \\
	X\ar[bend left=15]{ur}{\bar{g}}\ar[bend right=15]{dr}[swap]{\bar{f}} \ar[bend left=10]{rr}[pos=.2,description]{g}& & Y\\
	& P\ar[bend right=15]{ur}[swap]{\hat{f}}\ar[bend left=15, crossing over,from=uu,pos=.6,"\phi"]
	\end{tikzcd}
	\]
\end{point} 

\begin{theorem}\label{tame_is_model_cat}
	Let $\M$ be a model category.
	The following choices of weak equivalences, fibrations and cofibrations form a model structure on $\tamem$. 
	A morphism $g\colon X\to Y$  in $\tamem$ is a
	\begin{itemize}
		\item weak equivalence if $g^t\colon X^t\to Y^t$ is a weak equivalence for all $t$.
		\item fibration if $g^t\colon X^t\to Y^t$ is a fibration for all $t$.
		\item cofibration if $\hat{g}^t\colon Q^t\to Y^t$ (see \ref{point_factorisation}) is a cofibration for all $t$. 
	\end{itemize}
\end{theorem}

Due to tameness, to prove $g\colon X\to Y$ in $\tamem$ is a weak equivalence, or a fibration, or a cofibration, only finitely many verifications need to be performed. 
If $0=\tau_0<\cdots<\tau_k$ discretises both $X$ and $Y$, then  $g$ is a weak equivalence (respectively, a fibration) if and only if $g^{\tau_a}\colon X^{\tau_{a}}\to Y^{\tau_{a}}$ is a weak equivalence (respectively, a fibration) in $\M$ for any $a=0,\ldots,k$. 
Similarly, $g$ is a cofibration if and only if $\hat{g}^{\tau_a}\colon Q^{\tau_{a}}\to Y^{\tau_{a}}$ is a cofibration in $\M$ for any $a=0,\ldots,k$. 
It is important to realise however that for $g$ to  be a cofibration is it not enough for $g^t$ to be a cofibration for all $t$.

\begin{proposition}\label{char_cofibrant} 
	Let $\M$ be a model category. 
	\begin{enumerate}
		\item If $g\colon X\to Y$ is a cofibration in $\tamem$, then
		$g^t\colon X^t\to Y^t$ is a cofibration in $\M$ for any $t$ in $\posr$.
		\item An object $X$ in $\tamem$ is cofibrant if and only if $X^{0}$ is cofibrant and, for any $s<t$ in $\posr$, the transition morphism $X^{s<t}\colon X^{s}\to X^{t}$ is a cofibration in $\M$.
	\end{enumerate}
\end{proposition}

\begin{proof}
	\textit{1.}\quad
	Assume $g\colon X\hookrightarrow Y$ is a cofibration. Let  $0=\tau_0<\cdots<\tau_k$  be a sequence discretising both $X$ and $Y$.
	By definition $g^0=g^{\tau_0}=\hat{g}^{\tau_0}$ is a cofibration. 
	Since in a model category cofibrations  are preserved by compositions and taking  pushouts  along any morphism, the indicated arrows in the following commutative diagram are cofibrations for any $a=1,\ldots,k$:
	\[
	\begin{tikzcd}[column sep=1.2cm,row sep=18pt]
	X^{\tau_{a-1}} \ar[r, "X^{\tau_{a-1}<\tau_a}"]
	\ar[d, hook', "g^{\tau_{a-1}}"']
	& X^{\tau_{a}} \ar[ddr, hook, "g^{\tau_a}", bend left=15] 
	\ar[d, hook', "\bar{g}^{\tau_a}"']
	& [-0.4cm]
	\\
	Y^{\tau_{a-1}} \ar[rrd, bend right=13, "Y^{\tau_{a-1}<\tau_a}"'] \ar[r]
	& Q^{\tau_a} \ar[hook]{dr}[pos=0.2]{\hat{g}^{\tau_a}}
	\\
	[-0.6cm]
	& 
	& Y^{\tau_{a}}
	\end{tikzcd}
	\]
	Thus, for any $a$ in $[k]$, the morphism $g^{\tau_a}\colon X^{\tau_a}\to Y^{\tau_a}$ is a cofibration. 
	Tameness can be then used to conclude that $g^{t}$ is a cofibration for any $t$ in $\posr$.
	
	\noindent
	\textit{2.}\quad
	Let  $\emptyset\colon \posr\to\M$ be the initial object in $\tamem$, which is the constant functor whose value is the initial object in $\M$.
	Consider a morphism  $g\colon \emptyset\to X$ in $\tamem$.
	Let $0=\tau_{0}<\cdots<\tau_{k}$ be a sequence  discretising $X$.  
	Then $Q^{0}=\emptyset$ and $Q^{\tau_a}=X^{\tau_{a-1}}$ for $a>0$.
	Furthermore, $\hat{g}^{0}=(\emptyset \to X^0)$ and    $\hat{g}^{\tau_{a}}\colon Q^{\tau_{a}}=X^{\tau_{a-1}}\to X^{\tau_{a}}$ is the transition morphism in $X$ for $a>0$. The statement is then a direct consequence of the definition of a cofibration in $\tamem$.
\end{proof}

\begin{proof}[Proof of Theorem~\ref{tame_is_model_cat}]
	\noindent
	\ref{MC1}: This is a consequence of the fact that there is  a sequence that discretises all elements in a finite collection of tame functors.
	
	\noindent
	\ref{MC2} and \ref{MC3}: These follows from the fact that  $\M$ satisfies these axioms, and from the functoriality of the mediating morphism $\hat{g}$.
	
	\noindent
	\ref{MC4}: Consider a commutative square in $\tamem$:
	\[
	\begin{tikzcd}[row sep=13]
	X \arrow[hook']{d}[left]{\alpha} \arrow[r] 
	& E\arrow[two heads,"\beta"]{d}
	\\
	Y\arrow[r] 
	& B
	\end{tikzcd}
	\]
	where either 
	$\alpha$ is a cofibration and $\beta$ is a fibration and a weak equivalence, or $\alpha$ is a cofibration and a weak equivalence and $\beta$ is a fibration. 
	We need to show that there is a morphism ${\phi} \colon Y\to E$ which if added to the above square would make the obtained diagram commutative. 
	Let us choose a sequence $0=\tau_0<\cdots<\tau_k$ that discretises all functors in this square.
	We are going to define by induction on $a$ in $[k]$ morphisms ${\phi}^{\tau_a}\colon Y^{\tau_a}\to E^{\tau_a}$. 
	We then use this sequence to get the desired  ${\phi}\colon Y\to E$.
	
	Set $\phi^{0}\colon Y^{0}\to E^{0}$ to be any morphism in $\M$ that makes the following square commutative. 
	It exists by the axiom \ref{MC4} in $\M$.
	\[
	\begin{tikzcd}[row sep=13]
	X^0 \arrow[hook']{d}[left]{\alpha^0} \arrow[r] & E^0\arrow[d,"\beta^0"]
	\\
	Y^0\arrow[r]\arrow{ur}[description]{\phi^0} 
	& B^0
	\end{tikzcd}
	\]
	Assume $a\geq 1$ and that we have defined $\phi^{\tau_b}\colon Y^{\tau_b}\to E^{\tau_b}$ for $b<a$. 
	We can then form the following commutative diagram, where the indicated arrows are cofibrations by Proposition~\ref{char_cofibrant}.1: 	
	\[
	\begin{tikzcd}[column sep=2.5em,row sep=20pt]
	& [-0.3cm]
	& [-0.3cm] X^{\tau_{a-1}}\arrow{rrr} \arrow{rrr} \arrow[hook']{ddd}[description]{\alpha^{\tau_{a-1}}}\arrow{dll}
	& [-0.3cm]
	& [-0.3cm]
	& X^{\tau_a}\arrow{dll} 
	\arrow[hook,bend left=30]{ddd}{\alpha^{\tau_a}} 
	\arrow[hook']{dd}[left]{\bar{\alpha}^{\tau_a}}
	\\
	E^{\tau_{a-1}}\arrow[two heads]{ddd}[left]{\beta^{\tau_{a-1}}} 
	\arrow[crossing over]{rrr} 
	& 
	& 
	& E^{\tau_a}
	\\
	[-0.3cm]
	& 
	& 
	& 
	& 
	& Q^{\tau_a} \arrow[hook']{d}[description]{\hat{\alpha}^{\tau_a}} 
	\arrow[dotted]{ull}[description]{\phi'}
	\\
	[0.25cm]
	& 
	& Y^{\tau_{a-1}} \arrow{dll} \arrow{rrr} \arrow{urrr} \arrow{uull}{\phi^{\tau_{a-1}}}
	& 
	& 
	& Y^{\tau_a} \arrow{dll}
	\arrow[dotted,bend left=15,crossing over]{uull}[description,near start]{\phi^{\tau_a}} 
	\\
	B^{\tau_{a-1}} \arrow{rrr} 
	& 
	& 
	& B^{\tau_a}\arrow[from=uuu,crossing over, two heads, near start,swap,"\beta^{\tau_a}"]
	\end{tikzcd}
	\]
	All the horizontal arrows represent the transition morphisms, $\phi'\colon Q^{\tau_a}\to E^{\tau_a}$ is induced by the universal property of a pushout, and $\phi^{\tau_a}\colon Y^{\tau_a}\to E^{\tau_a}$ is any morphism that makes the following diagram commute, whose existence is guaranteed by axiom \ref{MC4}:
	\[
	\begin{tikzcd}[column sep=30pt,row sep=13]
	E^{\tau_a}\arrow[two heads]{d}[left]{\beta^{\tau_a}} 
	& Q^{\tau_a}\arrow[hook]{d}{\hat{\alpha}^{\tau_a}}\arrow[dotted]{l}[above]{\phi'}
	\\
	B^{\tau_a} 
	& Y^{\tau_a} \arrow{l}\arrow[dotted]{lu}[description,pos=0.4]{\phi^{\tau_a}}
	\end{tikzcd}
	\]	
	
	\noindent
	\ref{MC5}: Consider a morphism $g\colon A\to X$ in $\tamem$.
	Let us choose a sequence $0=\tau_0<\cdots<\tau_k$ that discretises both $A$ and $X$.
	By induction on $a=0,\ldots,k$, we are going to construct the appropriate factorisations $g^{\tau_a}=\beta^{\tau_a}\alpha^{\tau_a}$.  
	Set $\alpha^0\colon A^0\hookrightarrow Y^0$ and $\beta^0\colon Y^0\twoheadrightarrow X^0$ to be the factorisation of $g^0\colon A\to X$, where one of $\alpha^{0}, \beta^0$ is also a weak equivalence.
	Such a factorisation exists by \ref{MC5} in $\M$. 
	Assume $a\geq 1$ and that we have defined $\alpha^{\tau_b}\colon A^{\tau_b}\hookrightarrow Y^{\tau_b}$ and $\beta^{\tau_b}\colon Y^{\tau_b}\twoheadrightarrow X^{\tau_b}$ for $b<a$. 
	We can then define:
	\[
	Q^{\tau_a}:=
	\colim(
	\begin{tikzcd}[column sep=4em] A^{\tau_a} 
	& A^{\tau_{a-1}} \arrow{l}[swap]{A^{\tau_{a-1}<\tau_a}} 
	\arrow[hook]{r}{\alpha^{\tau_{a-1}}} 
	& Y^{\tau_{a-1}}
	\end{tikzcd}
	)
	\]
	and form the following commutative diagram:	
	\[
	\begin{tikzcd}[column sep=3em,row sep=15pt]
	A^{\tau_{a-1}} \arrow[bend left=20]{rrrr}[description]{g^{\tau_{a-1}}} 
	\ar[rr, hook, "\alpha^{\tau_{a-1}}"]
	\ar[dd, "A^{\tau_{a-1}<\tau_a}"']
	& 
	& Y^{\tau_{a-1}} \arrow[two heads]{rr}{\beta^{\tau_{a-1}}}\arrow{dd} 
	& 
	& X^{\tau_{a-1}} \arrow{dd}[right]{X^{\tau_{a-1}<\tau_a}}
	\\
	& 
	& 
	& Y^{\tau_a} \arrow[dotted, two heads,bend left=10]{dr}{\beta^{\tau_a}}
	\\
	A^{\tau_a} \arrow[hook]{rr}{\alpha'}
	\arrow[bend right=15]{rrrr}[description]{g^{\tau_a}}
	\arrow[hook,bend left=15,dotted,crossing over]{rrru}{\alpha^{\tau_a}}
	& 
	& Q^{\tau_a} \arrow{rr}{\beta'} \arrow[hook]{ur}{\alpha''} 
	& 
	& X^{\tau_a}
	\end{tikzcd}
	\] 
	where the left square is pushout and $\beta'\colon Q^{\tau_a}\to X^{\tau_a}$ is induced by the universal property of the pushout.
	The morphisms $\alpha''\colon Q^{\tau_a}\to Y^{\tau_a}$ and $\beta^{\tau_a}\colon Y^{\tau_a}\to X^{\tau_a}$ form the appropriate factorisation of $\beta'$ into the composition of either a cofibration which is a weak equivalence and a fibration, or a cofibration and a fibration which is a weak equivalence.
	Set $\alpha^{\tau_a}\colon A^{\tau_a}\to Y^{\tau_a}$ to be the composition $\alpha''\alpha'$, and $Y^{\tau_{a-1}<\tau_{a}}\colon Y^{\tau_{a-1}}\to Y^{\tau_a}$ to be the composition of $Y^{\tau_{a-1}}\to Q^{\tau_a}$ and $\alpha''\colon P^{\tau_a}\to Y^{\tau_a}$.
	Define $Y$ to be the Kan extension along $0=\tau_0<\cdots<\tau_k$ of the sequence $\{Y^{\tau_{a-1}<\tau_{a}}\}_{a=1,\ldots,k}$ (see \ref{point_kan_extension}). 
	Let $\alpha\colon A\to Y$ and $\beta\colon Y\to X$ be the natural transformations induced by the sequences of morphisms $\{\alpha^{\tau_a}\colon A^{\tau_a}\to Y^{\tau_a}\}_{a=0,\ldots,k}$ and $\{\beta^{\tau_a}\colon Y^{\tau_a}\to X^{\tau_a}\}_{a=0\ldots,k}$. 
	By construction, $\alpha$ is a cofibration and $\beta$ is a fibration. 
	Furthermore, depending on the choice of the factorisations of $\beta'\colon Q^{\tau_a}\to X^{\tau_a}$, either $\alpha$ or $\beta$ is a weak equivalence.
\end{proof}

\begin{point}\label{point_min_factorisation_tamem}
	{\em Minimal factorisations}.
	Assume $\M$ satisfies the minimality axiom.
	Consider a morphism $g\colon A\to X$ in $\tamem$. 
	Perform the same constructions as in the proof of \ref{MC5} but instead of taking arbitrary factorisations consider the minimal ones. 
	In step zero, we take morphisms $\alpha^0\colon X^0\hookrightarrow Y^0$ and $\beta^0\colon Y^0\xrightarrowdbl{\sim}X^0$ that form a minimal factorisation of $g^0\colon A\to X$.
	Analogously, in the $a$-th step we take morphisms $\alpha''\colon Q^{\tau_a}\hookrightarrow Y^{\tau_a}$ and $\beta^{\tau_a}\colon Y^{\tau_a}\xrightarrowdbl{\sim} X^{\tau_a}$ which form a minimal  factorisation of $\beta'\colon Q^{\tau_a}\to X^{\tau_a}$. 
	We claim that the obtained morphisms $\alpha\colon A\hookrightarrow Y$ and $\beta\colon Y\xrightarrowdbl{\sim} X$ form a minimal factorisation of $g\colon A\to X$.
	We just proved:
\end{point}

\begin{theorem}\label{sdvgsdfhsfgn}
	If the model category $\M$ satisfies the minimal factorisation axiom, then so does $\tamem$.
\end{theorem}

\begin{corollary}
	If the model category $\M$  satisfies the minimal factorisation axiom, then so does $\text{\rm tame}([0,\infty)^k,\mathcal{M})$
	for any $k=1,2,\ldots$.
\end{corollary}

\section{Chain complexes of vector spaces}\label{section_chain_complexes}

Let $K$ be a field and ${\mathbf N}=\{0,1,\ldots\}$ the set of natural numbers. 
A (non-negatively graded) chain complex of $K$-vector spaces is a sequence of linear functions $X=\{\delta_n\colon X_{n+1}\to X_n\}_{n\in{\mathbf N}}$ of $K$-vector spaces, called {\bf differentials}, such that $\delta_{n}\delta_{n+1}=0$ for all $n$ in ${\mathbf N}$. 
In the notation of the differentials we often ignore their indexes and simply denote them by $\delta$, or $\delta_X$ to indicate which chain complex is considered.

A chain complex $X$ is called {\bf compact} if $\bigoplus_{n\in{\mathbb N}} X_n$ is finite dimensional \cite{adamek_rosicky}. 
This happens if and only if $X_n$ is  finite dimensional for all $n$ and $X_n$ is trivial for $n\gg 0$.

\begin{point}{\em Homology.}\label{point_chain_complexes}
	The following vector spaces are called respectively the $n$-th cycles and the $n$-th boundaries of $X$:
	\[
	Z_nX:=
	\begin{cases}
	X_0 
	& \text{ if }n=0
	\\
	\text{ker}(\delta_{n-1}\colon X_n\to X_{n-1})&\text{ if }n\geq 1
	\end{cases}
	\ \ \ \ \ \ \ 
	B_nX:=\text{im}(\delta_{n}\colon  X_{n+1}\to X_n)
	\]
	
	Since $\delta_{n}\delta_{n+1}=0$, the $n$-th boundaries  $B_nX$ is a vector subspace of the $n$-th cycles $Z_nX$.
	The quotient $Z_nX/B_nX$ is called the $n$-th {\bf homology} of $X$ and is denoted by $H_nX$.
	We write $ZX$, $BX$ and $HX$ to denote the non-negatively  graded vector spaces $\{Z_nX\}_{n\in {\mathbf N}}$, $\{B_nX\}_{n\in {\mathbf N}}$, and $\{H_nX\}_{n\in {\mathbf N}}$.
\end{point}

\begin{point}{\em Model structure.}\label{point_model_cat_ch}
	A morphism of chain complexes $g\colon X\to Y$ is a sequence of linear functions $\{g_n\colon X_n\to Y_n\}_{n\in {\mathbf N}}$ such that $g_n \delta_{X}=\delta_Y g_{n+1}$ for all $n$. 
	Such a morphism maps boundaries and cycles in $X$ to  boundaries and cycles in $Y$. 
	The induced map on homologies is denoted by $Hg\colon HX\to HY$.
	If $Hg\colon HX\to HY$ is an isomorphism, then $g$ is a weak equivalence.
	If $g_n\colon X_n\to Y_n$ is an epimorphism for all $n\geq 1$
	(no assumption is made for $n=0$), 
	then $g$ is a fibration. 
	If $g_n\colon X_n\to Y_n$ is a monomorphism for all $n\geq 0$, 
	then $g$ is a cofibration.
	This choice of weak equivalences, fibrations and cofibrations defines a model structure on the category of chain complexes, denoted by $\Ch$
	(see \cite{dwyerspalinski, quillen}). 
	Consider the full subcategory of $\Ch$ given by compact chain complexes.
	The same choices of weak equivalences, fibrations, and cofibrations, as for $\Ch$, define a model structure on such a subcategory, denoted by $\ch$.
\end{point}

\begin{point}{\em Suspension.}\label{point_suspension}
	The {\bf suspension} of a chain complex $X$, denoted by $SX$ is a chain complex such that:
	\[
	\delta_{n}\colon (SX)_{n+1}\to (SX)_n=\begin{cases} X_0\to 0 & \text{ if } n=0
	\\
	-\delta_{n-1}\colon X_n\to  X_{n-1} & \text{ if }n>0\end{cases}
	\]
	Analogously, the suspension of a morphism $g\colon X\to Y$ of chain complex is a morphism $Sg\colon SX\to SX$ such that:
	\[
	(Sg)_n\colon (SX)_n\to (SY)_n=
	\begin{cases}
	0\colon 0\to 0 & \text{ if } n=0 \\ 
	g_{n-1}\colon X_{n-1}\to Y_{n-1} & \text{ if } n>0\
	\end{cases}
	\]
	The assignment $g\mapsto Sg$ is a functor denoted by $S\colon \Ch\to\Ch$.
	
	Note that $H_0(SX)=0$ and $H_{n}SX$ is isomorphic to $H_{n-1}X$ for all $n>0$. 
	Furthermore, if $f$ is a cofibration or a weak equivalence, then so is $Sf$, and if $f$ is a fibration, then 
	$Sf$ is a fibration if and only if $f_{0}$ is an epimorphism.

	The {\bf desuspension} of a chain complex $X$, denoted by $S^{-1}X$, is  a chain complex such that:
	\[
	\delta_{n}\colon (S^{-1}X)_{n+1}\to (S^{-1}X)_n=
	\begin{cases} 
	-\delta_1\colon X_2\to  Z_1X  
	& \text{ if } n=0
	\\
	-\delta_{n+1}\colon X_{n+2}\to  X_{n+1} 
	& \text{ if }n>0\end{cases}
	\]
	
	Analogously, the desuspension of a morphism $g\colon X\to Y$ of chain complex is a morphism $S^{-1}g\colon S^{-1}X\to S^{-1}X$ such that:
	\[
	(S^{-1}g)_n\colon (S^{-1}X)_n\to (S^{-1}Y)_n=
	\begin{cases}
	g_1\colon Z_1X \to Z_1Y  & \text{ if } n=0 \\ 
	g_{n+1}\colon X_{n+1}\to Y_{n+1} & \text{ if } n>0\
	\end{cases}
	\]
	The assignment $g\mapsto S^{-1}g$ is a functor denoted by $S^{-1}\colon\Ch\to\Ch$.
	
	Note that $H_{n}S^{-1}X$ is isomorphic to $H_{n+1}X$.
	If $f$ is a fibration, cofibration or a weak equivalence, then so is $S^{-1}f$. 
	Furthermore, $S^{-1}SX$ is isomorphic to $X$, and
	$SS^{-1}X$ is isomorphic to $X$ if and only if $X_0=0$. 	
\end{point}

We now provide some explicit constructions of chain complexes   used essentially in~\ref{point_standard_dec_min_rep_ch} to compute the standard decomposition and the minimal representative in $\ch$. 

\begin{point}\label{point_cofiber}
	{\em Cofiber sequences.}
	Let $f\colon X\to Y$ be a morphism of chain complexes. 
	Define a chain complex $Cf$, called the \textbf{cofiber} of $f$, a cofibration $i\colon Y\hookrightarrow Cf$, and a fibration $p\colon Cf\twoheadrightarrow SX$, as follows:
	\[
	\begin{tikzcd}[column sep=3.5em, row sep=18pt, ampersand replacement=\&]
	Y\arrow[hook]{d}{\scriptscriptstyle  i} 
	\& [-1cm]{\cdots}\ar[r, "\delta_{Y}"] 
	\& Y_3\arrow{r}{\scriptscriptstyle \delta_Y} 
	\arrow{d}{\left[\begin{smallmatrix} \scriptscriptstyle \id \\ \scriptscriptstyle 0 \end{smallmatrix}\right]}
	\& Y_2\arrow{r}{\scriptscriptstyle \delta_Y}
	\arrow{d}{\left[\begin{smallmatrix} \scriptscriptstyle \id \\ \scriptscriptstyle 0 \end{smallmatrix}\right]}
	\& Y_1
	\arrow{r}{\scriptscriptstyle \delta_Y} 
	\arrow{d}{\left[\begin{smallmatrix} \scriptscriptstyle \id \\ \scriptscriptstyle 0 \end{smallmatrix}\right]}
	\& Y_0
	\arrow{d}{\scriptscriptstyle \id}
	\\
	Cf\arrow[two heads]{d}{\scriptscriptstyle  p} 
	\& \cdots \arrow{r}{ \left[\begin{smallmatrix}\scriptscriptstyle \delta_Y &\scriptscriptstyle f\\ \scriptscriptstyle 0& \scriptscriptstyle -\delta_X \end{smallmatrix}\right]} 
	\& Y_3\oplus X_2 
	\arrow{d}{\left[\begin{smallmatrix}\scriptscriptstyle 0 & \scriptscriptstyle \id \end{smallmatrix}\right]}
	\arrow{r}{ \left[\begin{smallmatrix}\scriptscriptstyle \delta_Y &\scriptscriptstyle f \\ \scriptscriptstyle 0& \scriptscriptstyle -\delta_X\end{smallmatrix}\right]} 
	\& Y_2\oplus X_1 
	\arrow{d}{\left[\begin{smallmatrix}\scriptscriptstyle 0 & \scriptscriptstyle \id \end{smallmatrix}\right]}
	\arrow{r}{ \left[\begin{smallmatrix}\scriptscriptstyle \delta_Y &\scriptscriptstyle f \\ \scriptscriptstyle 0& \scriptscriptstyle -\delta_X\end{smallmatrix}\right]} 
	\& Y_1\oplus X_0 
	\arrow{d}{\left[\begin{smallmatrix}\scriptscriptstyle 0 & \scriptscriptstyle \id \end{smallmatrix}\right]}
	\arrow{r}{\left[\begin{smallmatrix} \scriptscriptstyle \delta_Y &\scriptscriptstyle f \end{smallmatrix}\right] } 
	\& Y_0\arrow{d}{\scriptscriptstyle 0}
	\\
	SX 
	\& \cdots\arrow{r}{\scriptscriptstyle -\delta_X} 
	\& X_2\arrow{r}{\scriptscriptstyle -\delta_X} 
	\& X_1\arrow{r}{\scriptscriptstyle -\delta_X} 
	\& X_0\arrow{r}{} 
	\& 0
	\end{tikzcd}
	\]
	The cofibration $i$ and the fibration $p$ form an exact sequence, called the \textbf{cofiber sequence} of $f$:
	\[
	\begin{tikzcd}
	0\arrow{r} 
	& Y\arrow[hook]{r}{i} 
	& Cf\arrow[two heads]{r}{p} 
	& SX\arrow{r} 
	& 0
	\end{tikzcd}
	\]
	
	Consider two maps of chain complexes $f\colon X\to Y$ and $g\colon W\to Z$. 
	Each of them leads to a cofiber sequence.
	A natural transformation between these exact sequences is by definition
	a triple of morphisms of chain complexes
	$S\alpha\colon SX\to SW$, $\beta\colon Y\to Z$ and $\gamma\colon Cf\to Cg$ which make the following diagram commute:
	\[\begin{tikzcd}[row sep=13pt]
	0\arrow{r} 
	& Y\arrow[hook]{r}{i} \ar{d}{\beta}
	& Cf\arrow[two heads]{r}{p} \ar{d}{\gamma}
	& SX\arrow{r} \ar{d}{S\alpha}
	& 0
	\\
	0\arrow{r} 
	& Z\arrow[hook]{r}{i} 
	& Cg\arrow[two heads]{r}{p} 
	& SW\arrow{r} 
	& 0
	\end{tikzcd}
	\]
	Commutativity of this diagram has two consequences. 
	First, $\gamma$ is of the form:
	\[
	\gamma_n=
	\begin{cases} 
	\beta_0\colon Y_0\to Z_0 &\text{ if } n=0
	\\
	\left[\begin{smallmatrix}
	\beta_{n} & h_{n-1} \\ 
	0 &\alpha_{n-1}
	\end{smallmatrix}\right]
	\colon Y_n\oplus X_{n-1}\to Z_n\oplus W_{n-1}&\text{ if } n>0
	\end{cases}
	\]
	Second, the sequence of linear functions $h=\{h_n\colon X_n\to Z_{n+1}\}_{n\geq 0}$ satisfies the equation $\beta f-g\alpha=\delta_Z h+h\delta_X$, which means that $h$ is a homotopy between $\beta f$ and $g\alpha$. 
	It follows that the set of natural transformations between the two cofiber sequences is in bijection with the set of triples consisting of  morphisms $\alpha\colon X\to W$ and $\beta\colon Y\to Z$, and a homotopy $h$ between $\beta f$ and $g\alpha$. We illustrate such a triple in form of a diagram:
	\[
	\begin{tikzcd}[row sep=13pt]
	X \ar[r, "f"] \ar[d, "\alpha"'] \ar[Rightarrow]{dr}{h} 
	& Y \ar{d}{\beta} 
	\\
	W \ar{r}[description]{g}
	& Z
	\end{tikzcd}
	\]
	The symbol $C(\alpha,\beta,h)\colon Cf\to Cg$ denotes 
	the morphism $\gamma \colon Cf\to Cg$, corresponding to this triple $(\alpha,\beta,h)$.
	
	In the case $h=0$, such  diagrams corresponds to a commutative squares:
	\[
	\begin{tikzcd}[row sep=13pt]
	X\ar{r}{f}\ar{d}[swap]{\alpha} \ar[Rightarrow]{dr}{0}
	&Y\ar{d}{\beta} 
	\\
	W\ar{r}[description]{g} 
	&Z
	\end{tikzcd}
	\qquad = \qquad 
	\begin{tikzcd}[row sep=13pt]
	X\ar{r}{f}\ar{d}[swap]{\alpha} 
	& Y\ar{d}{\beta}
	\\
	W\ar{r}{g} 
	& Z
	\end{tikzcd}
	\]
	In this case, the corresponding morphism between the cofibers is denoted simply by $C(\alpha,\beta)\colon Cf\to Cg$.
	
	In the case the differentials $\delta_Z$ and $\delta_X$ are trivial (in all degrees), the following implication holds (homotopy commutative square is commutative):
	\[
	\begin{tikzcd}[row sep=13pt]
	X\ar{r}{f}\ar{d}[swap]{\alpha}\ar[Rightarrow]{dr}{h} 
	& Y\ar{d}{\beta} 
	\\
	W\ar{r}[description]{g} 
	& Z
	\end{tikzcd}
	\quad \text{implies} \quad
	\begin{tikzcd}[row sep=13pt]
	X\ar{r}{f}\ar{d}[swap]{\alpha}\ar[Rightarrow]{dr}{0} 
	& Y\ar{d}{\beta}
	\\
	W\ar{r}[description]{g} 
	& Z
	\end{tikzcd}
	\]
\end{point}
\begin{point}{\em Comparison morphism.}\label{adgsdgfhdfghn}
	Let $f\colon X\to Y$ be a morphism of chain complexes. Consider the quotient morphism $\text{q}\colon Y\to Y/f(X)$ and define the \textbf{comparison morphism} $Cf\to Y/f(X)$ to be:
	\[\begin{tikzcd}[column sep=3.5em, row sep=13pt, ampersand replacement=\&]
	Cf\arrow[two heads]{d}
	\& 
	\cdots \arrow{r}{ \left[\begin{smallmatrix}\scriptscriptstyle \delta_Y &\scriptscriptstyle f\\ \scriptscriptstyle 0& \scriptscriptstyle -\delta_X \end{smallmatrix}\right]} 
	\& 
	Y_2\oplus X_1 
	\arrow{d}{\left[\begin{smallmatrix}\scriptscriptstyle  q & \scriptscriptstyle 0 \end{smallmatrix}\right]}
	\arrow{r}{ \left[\begin{smallmatrix}\scriptscriptstyle \delta_Y &\scriptscriptstyle f \\ \scriptscriptstyle 0& \scriptscriptstyle -\delta_X\end{smallmatrix}\right]} 
	\& 
	Y_1\oplus X_0 
	\arrow{d}{\left[\begin{smallmatrix}\scriptscriptstyle q& \scriptscriptstyle 0 \end{smallmatrix}\right]}
	\arrow{r}{\left[\begin{smallmatrix} \scriptscriptstyle \delta_Y &\scriptscriptstyle f \end{smallmatrix}\right] } 
	\& 
	Y_0\arrow{d}{\scriptscriptstyle q}
	\\
	Y/f(X)	\&\cdots\ar{r}{\delta}    \& (Y/f(X))_2\ar{r}{\delta}  \& (Y/f(X))_1\ar{r}{\delta} \&(Y/f(X))_0
	\end{tikzcd}\]
	If $f$ is a monomorphism, then the comparison morphism $Cf\twoheadrightarrow Y/f(X)$ is a weak equivalence. 
\end{point}

\begin{point}{\em Factorisation.}
	The complex $C\id_X$ is also denoted by $CX$ and called the {\bf cone} on $X$. 
	Explicitly, the cofibration
	$i\colon X\hookrightarrow CX$ is given by:	
	\[
	\begin{tikzcd}[column sep=3.5em, row sep=15pt, ampersand replacement=\&]
	X\arrow[hook]{d}{\scriptscriptstyle  i} 
	\& [-1cm]{\cdots}\ar[r, "\delta_{X}"]
	\& X_3\arrow{r}{\scriptscriptstyle \delta_X} 
	\arrow{d}{\left[\begin{smallmatrix} \scriptscriptstyle \id \\ \scriptscriptstyle 0\end{smallmatrix}\right]}
	\& X_2\arrow{r}{\scriptscriptstyle \delta_X}
	\arrow{d}{\left[\begin{smallmatrix} \scriptscriptstyle \id \\ \scriptscriptstyle 0\end{smallmatrix}\right]}
	\& X_1
	\arrow{r}{\scriptscriptstyle \delta_X} 
	\arrow{d}{\left[\begin{smallmatrix} \scriptscriptstyle \id \\ \scriptscriptstyle 0\end{smallmatrix}\right]}
	\& X_0
	\arrow{d}{\scriptscriptstyle \id}
	\\
	CX \&{\cdots}  \arrow{r}{ \left[\begin{smallmatrix}\scriptscriptstyle \delta_X &\scriptscriptstyle \id \\ \scriptscriptstyle 0 & \scriptscriptstyle -\delta_X\end{smallmatrix}\right]} \& X_3\oplus X_2 
	\arrow{r}{ \left[\begin{smallmatrix}\scriptscriptstyle \delta_X &\scriptscriptstyle \id \\ \scriptscriptstyle 0 & \scriptscriptstyle -\delta_X\end{smallmatrix}\right]} 
	\& X_2\oplus X_1 
	\arrow{r}{ \left[\begin{smallmatrix}\scriptscriptstyle \delta_X &\scriptscriptstyle \id \\ \scriptscriptstyle 0 & \scriptscriptstyle -\delta_X\end{smallmatrix}\right]} 
	\& X_1\oplus X_0 
	\arrow{r}{\left[\begin{smallmatrix} \scriptscriptstyle \delta_X &\scriptscriptstyle \id \end{smallmatrix}\right] } 
	\& X_0
	\end{tikzcd}
	\]
	Note that $HCX=0$.
	
	The complex $S^{-1}CX$ is also denoted by $PX$ and called the {\bf path complex} on $X$. 
	We also use the symbol $p\colon PX\to X$ to denote the fibration given by the desuspension  $S^{-1}p\colon S^{-1}CX\to S^{-1}SX=X$. 
	Explicitly:
	\[
	\begin{tikzcd}[column sep=3.5em, row sep=15pt, ampersand replacement=\&]
	PX\arrow[two heads]{d}{\scriptscriptstyle  p} 
	\& [-1cm]\cdots \arrow{r}{ \left[\begin{smallmatrix}\scriptscriptstyle -\delta_X & \scriptscriptstyle \id \\ \scriptscriptstyle 0 & \scriptscriptstyle \delta_X\end{smallmatrix}\right]} 
	\& X_4\oplus X_3 
	\arrow[two heads]{d}{\left[\begin{smallmatrix}\scriptscriptstyle 0 & \scriptscriptstyle \id \end{smallmatrix}\right]}
	\arrow{r}{ \left[\begin{smallmatrix}\scriptscriptstyle -\delta_X & \scriptscriptstyle \id \\ \scriptscriptstyle 0 & \scriptscriptstyle \delta_X\end{smallmatrix}\right]} 
	\& X_3\oplus X_2 
	\arrow[two heads]{d}{\left[\begin{smallmatrix}\scriptscriptstyle 0 & \scriptscriptstyle \id \end{smallmatrix}\right]}
	\arrow{r}{ \left[\begin{smallmatrix}\scriptscriptstyle -\delta_X &\scriptscriptstyle \id \\ \scriptscriptstyle 0 & \scriptscriptstyle \delta_X\end{smallmatrix}\right]} 
	\& X_2\oplus X_1 
	\arrow[two heads]{d}{\left[\begin{smallmatrix}\scriptscriptstyle 0 & \scriptscriptstyle \id \end{smallmatrix}\right]}
	\arrow{r}{\left[\begin{smallmatrix} \scriptscriptstyle - \delta_X & \scriptscriptstyle1 \end{smallmatrix}\right] } 
	\& X_1\arrow[two heads]{d}{\scriptscriptstyle \delta_X}
	\\
	X 
	\& \cdots\arrow{r}{\scriptscriptstyle \delta_X} 
	\& X_3\arrow{r}{\scriptscriptstyle \delta_X} 
	\& X_2\arrow{r}{\scriptscriptstyle \delta_X} 
	\& X_1\arrow{r}{ \delta_X} 
	\& X_0
	\end{tikzcd}
	\]
	Note that $HPX=0$.
	
	Since $HPX=0=HCX$, the fibration $p\colon PX\twoheadrightarrow X$ and the cofibration $i\colon X\hookrightarrow CX$ fit into the following factorisations of the morphisms $0\to X\to 0$:
	\[
	\begin{tikzcd}[row sep=2pt]
	& PX\arrow[two heads,bend left=10]{rd}{p} 
	& 
	& CX\arrow[two heads,bend left=10]{rd}{\sim}
	\\
	0\arrow[bend left=10,hook]{ru}{\sim} \arrow{rr} 
	& 
	& X\arrow{rr}\arrow[hook,bend left=10]{ru}{i} 
	& 
	& 0
	\end{tikzcd}
	\]
	
	These morphisms $i$ and $p$ can be used to construct explicit  factorisations of arbitrary morphisms in $\Ch$, whose existence is guaranteed by axiom \ref{MC5}: any $g\colon X\to Y$ fits into a commutative diagram:
	\[
	\begin{tikzcd}[ampersand replacement=\&,row sep=3pt]
	\& X\oplus PY 
	\arrow[two heads,bend left=10]{rd}{\left[\begin{smallmatrix}  g & p \end{smallmatrix}\right]}
	\\
	X \arrow{rr}[description]{g}
	\ar[hook,bend left=10]{ru}{\left[\begin{smallmatrix} \id \\ 0 \end{smallmatrix}\right]}[swap]{\sim} 
	\ar[hook',bend right=10]{rd}[swap]{\left[\begin{smallmatrix}i\\g \end{smallmatrix}\right]}
	\& 
	\& Y
	\\
	\& CX\oplus Y 
	\ar[two heads,bend right=10]{ru}[swap]{\left[\begin{smallmatrix}0& 1 \end{smallmatrix}\right]}{\sim} 
	\end{tikzcd}
	\]
	These factorisations are natural, however in general not minimal (see Definition~\ref{def_minimal_morphism}).  
	To obtain minimal factorisations we cannot perform natural constructions and we will be forced to make some choices. 
\end{point}

\begin{point}\label{point_graded_vect}
	{\em Graded vector spaces.} 
	A (non-negatively) graded $K$-vector space is by definition a sequence of $K$-vector spaces $V=\{V_n\}_{n\in{\mathbf N}}$. 
	Such a graded vector space is concentrated in degree $k$ if $V_n=0$ for all $n\not= k$. 
	Graded vector spaces concentrated in degree $0$ are identified   with vector spaces. 
	
	Let $V=\{V_n\}_{n\in{\mathbf N}}$ be a graded $K$-vector space.
	The same symbol $V$ is also used to denote the chain complex $\{0\colon V_{n+1}\to V_n\}_{n\in{\mathbf N}}$ with the trivial differentials. 
	In this case, $HV=V$ and hence any weak equivalence $\phi\colon V\to V$ is an isomorphism. 
	In fact, an arbitrary chain complex $X$ is minimal (see Definition~\ref{def_minimal_object}) if and only if all its differentials are trivial. 
	More generally any cofibration $\alpha\colon X\hookrightarrow Y$ for which the chain complex $Y/\alpha(X)$ has all trivial differentials satisfies the following minimality condition: 
	{\em any weak equivalence $\phi\colon Y\to Y$ for which $\alpha \phi  =\alpha$ is an isomorphism.} 
	To see this consider a commutative diagram with exact rows:
	\[
	\begin{tikzcd}[row sep=10pt]
	0\ar{r} 
	& X\ar[hook]{r}{\alpha}\ar{d}[swap]{\id} 
	& Y\ar[two heads]{r}\ar{d}{\phi} 
	& Y/\alpha(X)\ar{r} \ar{d}
	& 0
	\\
	0\ar{r} 
	& X\ar[hook]{r}{\alpha} 
	& Y\ar[two heads]{r} 
	& Y/\alpha(X)\ar{r} 
	& 0
	\end{tikzcd}
	\]
	Using the long sequences of homologies for each row, we can conclude the morphism $Y/\alpha(X)\to Y/\alpha(X)$ is a weak equivalence and hence an isomorphism as $Y/\alpha(X)$ is assumed to have all differentials trivial.
	We can then use the exactness of the rows to get that $\phi$ is also an isomorphism.
	
	To denote the $n$-fold suspension of $K$ we use the symbol $\Sp$.
	Explicitly, $\Sp$ is the chain complex concentrated in degree $n$ such that $(\Sp)_n=K$. 
	For example $\Sp[0]=K$. 
	The complex $\Sp$ is called the $n$-th {\bf sphere}. 
	The cone $C\Sp$ is denoted by $\D[n+1]$ and called the $(n+1)$-st {\bf disk}. 
	Explicitly:
	\[
	(\D[n+1])_k=
	\begin{cases}
	K &\text{ if } k=n\text{ or } k=n+1
	\\
	0 & \text{ otherwise}
	\end{cases},
	\quad
	\delta_k=
	\begin{cases}
	\id &\text{ if } k=n 
	\\
	0 & \text{ otherwise}
	\end{cases}
	\]
\end{point}

\begin{point}{\em Standard decomposition and minimal representative.}\label{point_standard_dec_min_rep_ch}
	Let $X$ be a chain complex.
	Consider the morphisms $p\colon CBX\twoheadrightarrow SBX\leftarrow X:\!\delta_{X}$ (see \ref{point_cofiber}).  
	Axiom \ref{MC4} guarantees existence of a morphism $\phi\colon CBX\to X$ making the following diagram commutative:
	\[
	\begin{tikzcd}[row sep=13pt]
	0\ar[hook']{d}[swap]{\sim} \ar{r} 
	& X\ar[two heads]{d}{\delta} 
	\\
	CBX\ar[two heads]{r}[description]{p}\ar[dashed]{ru}[description]{\phi} 
	& SBX
	\end{tikzcd}
	\]
	
	The restriction of any such $\phi$  to $i\colon BX\hookrightarrow CBX$ is the standard inclusion $BX\hookrightarrow  ZX\hookrightarrow X$. 
	This can be seen by looking at the long exact sequences of homologies applied to the rows in the following commutative diagram:
	\[
	\begin{tikzcd}[row sep=13pt]
	0\ar{r} 
	& BX\ar[hook]{r}{i}\ar{d} 
	& CBX\ar[two heads]{r}{p}\ar{d}{\phi} 
	& SBX\ar{r} \ar{d}{\id}
	& 0
	\\
	0\ar{r} 
	& ZX\ar[hook]{r} 
	& X\ar[two heads]{r}{\delta} 
	& SBX\ar{r} 
	& 0
	\end{tikzcd}
	\]
	The morphism $\phi$ leads therefore to a pushout square (in particular $\phi$ is a cofibration):
	\[
	\begin{tikzcd}[row sep=13pt]
	BX \ar[hook]{r}{i}\ar[hook']{d} 
	& CBX \ar[hook]{d}{\phi}
	\\
	ZX\ar[hook]{r} 
	& X
	\end{tikzcd}
	\]
	Since considered coefficients are in a field and all the differentials in $BX$, $ZX$ and $HX$ are trivial, there is a morphism $s\colon HX\to ZX$, whose composition with the quotient $ZX\twoheadrightarrow HX$ is $\id_{HX}$. 
	For any such $s$, the morphism  $\begin{bmatrix}i&s\end{bmatrix}\colon BX\oplus HX\to ZX$ is an isomorphism.  
	It follows that so is the morphism $\begin{bmatrix}\phi&s\end{bmatrix}\colon CBX\oplus HX\to X$, 
	where the symbol $s$ also denotes the composition of $s\colon HX\to ZX$ and the inclusion $ZX\hookrightarrow X$.
	We call $CBX\oplus HX$ the {\bf standard decomposition} of the chain complex $X$. 
	Since $CBX$ has trivial homology, the morphism $s\colon HX\to X$ is a weak equivalence and hence $HX$ is the minimal representative (see Definition~\ref{def_minimal_object}) of $X$.  
\end{point}

\begin{point}{\em Minimal factorisations.}
	Let $g\colon X\to Y$ be a morphism of chain complexes.  
	To construct its minimal factorisation (see Definition~\ref{def_minimal_morphism}) we perform the following steps:
	
	\begin{enumerate}
		\item Take the kernel $j\colon W\hookrightarrow X$ of $g\colon X\to Y$;
		\item Choose an isomorphism $W\xrightarrow{\simeq} CBW\oplus HW$ (see \ref{point_standard_dec_min_rep_ch});
		\item Consider the composition:
		\[
		\begin{tikzcd}[ampersand replacement=\&, column sep=3.3em]
		W\ar{r}{\simeq}\ar[hook',bend right=13]{rr}[description]{\alpha} 
		\& CBW\oplus HW\ar[hook]{r}{\left[\begin{smallmatrix} \id & 0 \\ 0 & i\end{smallmatrix}\right]} 
		\& CBW\oplus CHW
		\end{tikzcd}
		\]
		\item Use axiom \ref{MC4} to construct a morphism $\phi\colon X\to CBW\oplus CHW$, which fits into the
		following commutative diagram: 
		\[
		\begin{tikzcd}[row sep=13pt]
		W \ar[hook']{d}[swap]{j}\ar[hook]{r}{\alpha}
		& CBW\oplus CHW\ar[two heads]{d}{\sim}
		\\
		X \ar{r}\ar[dashed]{ru}[description]{\phi}
		& 0
		\end{tikzcd}
		\]
		\item The morphism
		$\left[\begin{smallmatrix} \phi\\ g\end{smallmatrix}\right]\colon X\to \left(CBW\oplus CHW\right)\oplus Y$ is then a cofibration.
	\end{enumerate}
	We are now ready to state:
\end{point}

\begin{proposition}
	The following factorisation is  minimal:
	\[
	\begin{tikzcd}[ampersand replacement=\&,row sep=2pt]
	\& \left(CBW\oplus CHW\right)\oplus Y 
	\arrow[two heads,bend left=10]{rd}{\left[\begin{smallmatrix} 0 & \id\end{smallmatrix}\right]}[swap]{\sim} 
	\\
	X \arrow{rr}[description]{g}
	\ar[hook,bend left=10]{ru}{\left[\begin{smallmatrix} \phi \\ g \end{smallmatrix}\right]}
	\& 
	\& Y
	\\
	\end{tikzcd}
	\]
\end{proposition}

\begin{proof}
	Let 
	$\psi=
	\left[\begin{smallmatrix} 
	\psi_{11} 
	& \psi_{12} 
	\\
	\psi_{21} 
	& \psi_{22}
	\end{smallmatrix}\right]
	\colon 
	\left(CBW\oplus CHW\right)\oplus Y
	\to 
	\left(CBW\oplus CHW\right)\oplus Y$ be a morphism
	making the following diagram commutative:
	\[
	\begin{tikzcd}[ampersand replacement=\&,row sep=20pt,column sep=2em]
	X \arrow[hook]{r}[description]{\left[\begin{smallmatrix} \phi\\g \end{smallmatrix}\right]} 
	\arrow[hook']{d}[swap]{\left[\begin{smallmatrix} \phi\\g \end{smallmatrix}\right]}
	\& \left(CBW\oplus CHW\right)\oplus Y 
	\arrow[two heads]{d}{\left[\begin{smallmatrix} 0 & \id \end{smallmatrix}\right]}[swap]{\sim} 
	\\
	\left(CBW\oplus CHW\right)\oplus Y 
	\arrow[two heads]{r}{\left[\begin{smallmatrix} 0 & \id\end{smallmatrix}\right]}[swap]{\sim} 
	\ar{ur}[description]{\psi}
	\& Y
	\end{tikzcd}
	\]
	Commutativity of the bottom triangle implies $\psi_{21}=0$ and $\psi_{22}=\id$.
	Since $W$ is the kernel of $g$, commutativity of the top triangle implies commutativity of:
	\[
	\begin{tikzcd}[row sep=13pt]
	W\ar[hook]{r}{\alpha}\ar[hook']{d}[swap]{\alpha}
	& CBW\oplus CHW
	\\
	CBW\oplus CHW\ar[bend right=10]{ru}[swap]{\psi_{11}}
	\end{tikzcd}
	\]
	The quotient $(CBW\oplus CHW)/\alpha(W)=SHW$ has all   differentials trivial. The morphism $\psi_{11}\colon CBW\oplus CHW\to CBW\oplus CHW$ is therefore an isomorphism (see \ref{point_graded_vect}). 
	It follows that so is $\psi=\left[\begin{smallmatrix} \psi_{11} &\psi_{12} \\ 0 & \id \end{smallmatrix}\right]$
\end{proof}	

We just proved that $\ch$ satisfies the minimal factorisation axiom.
By Theorem~\ref{sdvgsdfhsfgn}, it follows that $\tame$ also satisfies the minimal factorisation axiom, and thus any tame parametrised chain complex admits a minimal cover.
In Section~\ref{section_cofibrant_objects_tame}, we provide a characterisation of such minimal covers.

\section{Cofibrations in $\tame$}
\label{section_cofibrant_objects_tame}

In this section we discuss cofibrations in the model category $\tame$ as described in Theorem~\ref{tame_is_model_cat}.
According to Proposition~\ref{char_cofibrant}, an object $X$ in $\tame$ is cofibrant if and only if the transition $X^{s<t}\colon X^{s}\to X^{t}$ is a monomorphism (a cofibration in $\ch$) for every $s< t$ in $\posr$. 
This implies that if $Y$ is cofibrant, and $f\colon X\to Y$ is 
a monomorphism in $\tame$, then $X$ is also cofibrant.

For example, the following objects are cofibrant. 
They are parametrised by a natural number $n$ and an element in $\Omega:=\{(s,e)\in [0,\infty)\times [0,\infty]\ |\ s\leq e\}$, where $\left[0,\infty\right]$ is the poset of non-negative reals plus $\infty$.

\begin{definition}\label{def_interval_sphere}
	The Kan extensions (see \ref{point_kan_extension}) given by the data described in the following table are called {\bf interval spheres}:
	\begin{center}
		\begin{tabular}{c|c|c|c}
			Index & $n$, $s<e=\infty$ &$n$,  $s=e<\infty$ & $n$, $s<e<\infty$\\
			\hline
			Name  & $\intS{s,\infty}$ & $\intS{s,s}$ & $\intS{s,e}$\\
			\hline
			$k$ & $0$ & $0$ & $1$\\
			\hline
			Functor $[k]\to \ch$ & $\Sp[n]$ & $\D[n+1]$ & $i\colon \Sp[n]\hookrightarrow \D[n+1]$\\
			\hline
			Inclusion $[k]\subset\posr$ & $s$ & $s$ & $s<e$
		\end{tabular}
	\end{center}
\end{definition}
	
For example, $\intS[2]{5, \infty}\colon \posr\to \ch$ is a functor whose value at $t<5$ is $0$, and at $5\leq t$ is $\Sp[2]$.
Similarly, $\intS[2]{5, 5}\colon \posr\to \ch$ has value $0$ if $t<5$, and $\D[3]$ if $5\leq t$. 
The functor $\intS[2]{5, 7}\colon \posr\to \ch$
has three values: $0$ if $t<5$, $\Sp[2]$ if $5\leq t<7$, and $\D[3]$ if $7\leq t$. 
The transition morphisms in $\intS[2]{5, 7}$ are either the identities, or the inclusion $0\hookrightarrow \Sp[2]$ or the inclusion $i\colon \Sp[2]\hookrightarrow \D[3]$.

A morphism $\intS{s, e}\to 0$ is a weak equivalence if and only if $s=e$. 
Thus, interval spheres of type $\intS{s, s}$ are the only interval spheres for which the chain complex  $\intS{s, s}^t$ has trivial homology for every parameter $t$ in $\posr$.

The main result of this section is the structure theorem (compare with \cite{framed_morse, structural_filtered, related_work_decomposition}):

\begin{theorem}\thlabel{dec_theorem}
	\begin{enumerate}
		\item[(1)] Any cofibrant object in $\tame$ is isomorphic to
		a direct sum $\oplus_{i=1}^{l} \intS[n_{i}]{s_{i},e_{i}}$, where $l$ could possibly be $0$.
		\item[(2)] If $\oplus_{i=1}^{l} \intS[n_{i}]{s_{i},e_{i}} \cong \oplus_{j=1}^{l'} \intS[n'_{j}]{s'_{j},e'_{j}}$, then $l=l'$ and there is a permutation $\sigma$ of the set $\{1,\ldots, l\}$ such that $n_{i}=n'_{\sigma(i)}$,  $s_{i}=s'_{\sigma(i)}$, and $e_{i}=e'_{\sigma(i)}$ for any $i$.
	\end{enumerate}
\end{theorem}

To prove Theorem \ref{dec_theorem}, we first need to characterise cofibrations in $\tame$ and explain how to enumerate morphisms out of $\intS{s,e}$. 
We start with cofibrations:

\begin{proposition}\label{char_cofibrations}
	For every morphism $g\colon X\to Y$ in $\tame$, the following statements are equivalent:
	\begin{enumerate}
		\item $g$ is a cofibration;
		\item $g^t\colon X^t\to Y^t$ is a monomorphism for every $t$ in $\posr$, and $Y/g(X)$ is cofibrant;
		\item $g^t\colon X^t\to Y^t$ is a monomorphism for every $t$ in $\posr$ and, for all $s<t$ in $\posr$, the following is a pullback square:
		\[
		\begin{tikzcd}[row sep=12pt]
		X^{s} \ar[r,swap, "X^{s<t}"'] 
		\ar[d, hook',swap, "g^{s}"]
		& X^{t} \ar[d, hook, "g^{t}"]
		\\ 
		Y^{s}  \ar[r, "Y^{s<t}"]
		& Y^{t}
		\end{tikzcd}
		\]
	\end{enumerate}
\end{proposition}

\begin{proof} 
	In the proof, we utilise the following fundamental linear algebra statement.
	Consider commutative  diagrams   of  vector spaces:
	\[
	\begin{tikzcd}[row sep=12]
	V \ar[r, "\alpha_0"] \ar[d, "\alpha_1"'] 
	& W_0 \ar[d, "\beta_0"]
	\\
	W_1\ar{r}{\beta_1}
	& U
	\end{tikzcd}
	\]
	It leads to two vector spaces:
	\[
	P:=\lim(\begin{tikzcd}[ampersand replacement=\&] 
	W_1 \ar[r,"\beta_1"] \& U \& W_0 \ar[l, "\beta_0"']
	\end{tikzcd})
	\ \ \ \ \  \ \ 
	Q:=\colim(\begin{tikzcd}[ampersand replacement=\&] 
	W_1 \& V \ar[r, "\alpha_0"] \ar[l, "\alpha_1"'] \& W_0
	\end{tikzcd})
	\]
	and two linear functions $\alpha\colon V\to P$ and $\beta\colon Q\to U$ that make the following diagrams commutative, where the inside squares are respectively a pullback and a pushout:
	\[ 
	\begin{tikzcd}[column sep=0.9cm, row sep=15pt]
	V \ar[drr, bend left=16, "\alpha_0"] \ar[rd, "\alpha"]
	\ar[rdd, bend right=16, "\alpha_1"']
	& [-0.5cm]
	&
	\\
	[-0.4cm]
	& P\ar{r} \ar{d} 
	& W_0 \ar[d, "\beta_0"]
	\\
	& W_1 \ar[r,swap, "\beta_1"']
	& U
	\end{tikzcd} 
	\qquad \quad 
	\begin{tikzcd}[column sep=0.9cm, row sep=15pt]
	V \ar[r, "\alpha_0"] \ar[d, "\alpha_1"']
	& W_0 \ar[d] \ar[ddr, bend left=16, "\beta_0"]
	& [-0.4cm]
	\\
	W_1 \ar[r] \ar[drr, bend right=16, swap,"\beta_1"]
	& Q \ar[dr, pos=0.3,"\beta"]
	\\
	[-0.5cm]
	& 
	& U
	\end{tikzcd}
	\]
	Then $\alpha\colon V\to P$ is surjective if and only if   $\beta\colon Q\to U$ is injective. 
	\smallskip
	
	\noindent
	{\it 1}$\Rightarrow${\it 2}:\quad  The first part of {\it 2} follows from Proposition~\ref{char_cofibrant}.1, and the second from the fact that in a model category cofibrations are preserved by pushouts.
	\smallskip
	
	\noindent
	{\it 2}$\Rightarrow${\it 3}:\quad For all $s<t$ in $\posr$, we have the following commutative diagram, where the indicated arrows are cofibrations in $\ch$:
	\[
	\begin{tikzcd}[row sep=15pt]
	X^{s} \ar[d, "X^{s<t}"'] \ar[r, hook, "g^{s}"]
	& Y^{s} \ar[d, "Y^{s<t}"] \ar[r, two heads]
	& Y^{s}/g(X)^{s} \ar[d, hook]
	\\
	X^{t} \ar[r, hook, "g^{t}"] 
	& Y^{t} \ar[r, two heads]
	& Y^{t}/g(X)^{t} 
	\end{tikzcd}
	\]
	The pullback 
	$\lim(
	\hspace*{-0.15cm}\begin{tikzcd} 
	X^{t} \ar[r, hook, "g^{t}"] 
	& [-0.15cm] Y^{t} 
	& [0.15cm] Y^{s} \ar[l, "Y^{s<t}"'] 
	\end{tikzcd}\hspace*{-0.15cm})$ 
	is isomorphic to the kernel of the composition 
	$Y^s\twoheadrightarrow Y^{s}/g(X)^{s}\hookrightarrow Y^{t}/g(X)^{t}$.
	Since the second map is an inclusion, this pullback coincides with the kernel of $Y^s\twoheadrightarrow Y^s/g(X)^s$. 
	Consequently, the left square is a pullback.
	\smallskip	
	
	\noindent
	{\it 3}$\Rightarrow${\it 1}:\quad 
	Let $0=\tau_0<\cdots<\tau_k$ be a sequence that discretises both $X$ and $Y$. 
	According to Theorem~\ref{tame_is_model_cat} we need to show that $\hat{g}^{\tau_a}\colon P^{\tau_a}\to Y^{\tau_a}$ is a cofibration for all $a$.  
	This is a consequence of the linear algebra statement given at the beginning of the proof.
	Indeed, the square is a pullback, its mediating morphism is the identity and thus $\hat{g}^{\tau_a}$ is a monomorphism, and so a cofibration in $\ch$.
\end{proof}

\begin{point}\label{point_morphisms_out_interval_sphere_1}
	{\em Morphisms out of $\intS{s,\infty}$.}
	A morphism $g\colon \intS{s,\infty}\to X$ leads to a linear function $g_{n}^{s}\colon \intS{s,\infty}_{n}^{s}=K\to X_{n}^{s}$.
	Let $x:=g_{n}^{s}(1)$ in $X_{n}^{s}$.
	This element satisfies the equation $\delta(x)=0$, which means that $x$ belongs to the cycles $Z_{n}X^{s}$.
	Choosing an element in $Z_{n}X^{s}$ is all what is needed to describe a morphism out of $\intS{s,\infty}$.
	For any $x$ in $Z_{n}X^{s}$, there is a unique morphism $I(x)\colon\intS{s,\infty}\to X$ such that $x=I(x)_{n}^{s}(1)$.
	The association $g\mapsto g_{n}^{s}(1)$ describes a bijection (in fact a linear isomorphism) between the set of morphisms  $\intS{s,\infty}\to X$ and the set of cycles $Z_{n}X^{s}$.  
\end{point}

\begin{point}\label{point_morphisms_out_interval_sphere_2}
	{\em Morphisms out of $\intS{s,e}$.}
	Let $s\leq e<\infty$. 
	A morphism $g\colon \intS{s,e}\to X$ leads to two  functions $g_{n}^{s}\colon \intS{s,e}_{n}^{s}=K\to X_{n}^{s}$ and $g_{n+1}^{e}\colon \intS{s,e}_{n+1}^{e}=K\to X_{n+1}^{e}$.
	Define two elements $x:=g_{n}^{s}(1)$ in $X_{n}^{s}$ and  $y:=g_{n+1}^{e}(1)$ in $X_{n+1}^e$.
	These elements satisfy equations $\delta(x)=0$ and $X_n^{s\leq e}(x)=\delta(y)$.
	These equations contain all the information needed to describe a morphism out of $\intS{s,e}$.
	If $x$ in $X_{n}^{s}$ and $y$ in $X_{n+1}^{e}$ satisfy these  equations, then there is a unique morphism  $I(x,y)\colon \intS{s,e}\to X$ such that $x=I(x,y)_{n}^{s}(1)$ and $y=I(x,y)_{n+1}^{e}(1)$. 
	The association $g\mapsto\left(g_{n}^{s}(1),g_{n+1}^{e}(1)\right)$ 
	describes therefore a bijection between the set of morphisms $\intS{s,e}\to X$ and the pullback:
	\[
	\text{lim}(
	\hspace*{-0.15cm}\begin{tikzcd}[column sep=1.2cm]
	Z_{n}X^{s} \ar[r, "X_{n}^{s\leq e}"]
	& X_{n}^{e} 
	& [-0.3cm] X_{n+1}^{e} \ar[l, "\delta"']
	\end{tikzcd} 
	\hspace*{-0.15cm}
	)
	\]
	
	In the case $s=e$, this pullback can be identified with $X_{n+1}^{s}$. 
	Thus, the set of morphisms $\intS{s,s}\to X$ 
	is in bijection with $X_{n+1}^{s}$. 
	
	For  $X=\intS{s,s}$, the elements $1$ in $K=\intS{s,s}^s_n$ and $1$ in $K=\intS{s,s}^e_{n+1}$ satisfy the required equations. 
	The obtained morphism $\intS{s,e}\to \intS{s,s}$ is called the \textbf{standard inclusion}.
	Since not all of the transition morphisms of the quotient $\intS{s,s}/\intS{s,e}$ are monomorphisms, the standard inclusion is not a cofibration.
\end{point}

\begin{point}{\em Cofibrations out of $\intS{s,e}$.}
	In this paragraph we describe necessary and sufficient conditions for $g\colon \intS{s,e}\to X$ to be a cofibration.
	Since $\intS{s,e}$ is cofibrant, if there is such a cofibration, then $X$ has to be cofibrant. 
	Let us then make this assumption. 
	The transition morphisms in $X$ are therefore assumed to be cofibrations (see Proposition~\ref{char_cofibrant}).
	
	Choose $0=\tau_{0} <\cdots < \tau_{k}$  that discretises  $X$ and $\intS{s,e}$, and consider the diagrams:
	\[
	\begin{tikzcd}[column sep=2.1cm, row sep=18pt]
	\intS{s,e}^{\tau_{a-1}} \ar[d, "g^{\tau_{a-1}}"'] 
	\ar[r, hook, "{\intS{s,e}^{\tau_{a-1}<\tau_{a}}}"]
	& \intS{s,e}^{\tau_{a}} \ar[d, "g^{\tau_{a}}"]
	\\
	X^{\tau_{a-1}} \ar[r, hook, "X^{\tau_{a-1}<\tau_{a}}"]
	& X^{\tau_{a}}
	\end{tikzcd}\]
	By Proposition~\ref{char_cofibrations}, since $X$ is cofibrant, $g\colon \intS{s,e}\to X$ is a cofibration if and only if the diagram above is pullback for all $a=1,\dots, k$. 
	These diagrams are pullbacks if in every homological degree $h$ they are pullbacks of vector spaces. 
	They can fail to be so only if the transition morphism  $\intS{s,e}^{\tau_{a-1}<\tau_{a}}_h$ is not the identity, which happens in two cases: (i) $\tau_{a-1}<s=\tau_{a}$ and $h=n$, or (ii) $e<\infty$, $\tau_{a-1}<e=\tau_{a}$, and $h=n+1$. 
	In both of these case, the diagram above becomes:
	\[
	\begin{tikzcd}[column sep=1.9cm, row sep=15pt]
	0 \ar[r] \ar[d]
	& K \ar[d, "g^{\tau_{a}}_h"]
	\\
	X_{h}^{\tau_{a-1}} \ar[r, hook, "X^{\tau_{a-1}<\tau_{a}}_h"]
	& X_{h}^{\tau_{a}}
	\end{tikzcd}
	\]
	and hence it is a pullback if and only if $g^{\tau_{a}}_h(1)$ is not in the image of $X^{\tau_{a-1}<\tau_{a}}_h$.
	We have just proven:
\end{point}

\begin{proposition}\label{prop_cofibrations_from_interval_spheres}
	Let $X$ be an object in $\tame$. 
	
	\begin{enumerate}
		\item Let $n$ be a natural number, $s$ be in $\posr$, and  
		$x$ be in $Z_{n}X^{s}$. 
		Then the morphism $I(x)\colon \intS{s,\infty}\to X$ (see \ref{point_morphisms_out_interval_sphere_1}) is a cofibration if and only if $X$ is cofibrant and $x$ is not in the image of $X_{n}^{t<s}\colon X_{n}^{t}\to X_{n}^{s}$ for any $t<s$.
		\item Let $n$ be a natural number, $s\leq e$ be in $\posr$, and $x$ in $Z_{n}X^{s}$ and $y$ in $X^{e}_{n+1}$ be such that $X_{n}^{s\leq e}(x)=\delta(y)$. 
		Then $I(x,y)\colon \intS{s,e}\to X$ (see \ref{point_morphisms_out_interval_sphere_2}) is a cofibration if and only if $X$ is cofibrant, $x$ is not in the image of $X_{n}^{t<s}\colon X_{n}^{t}\to X_{n}^{s}$ for any $t<s$, and $y$ is not in the image of $X_{n+1}^{t<e}\colon X_{n+1}^{t}\to X_{n+1}^{e}$ for any $t<e$.
	\end{enumerate}
\end{proposition}

We are now ready to:
\begin{proof}[Proof of Theorem~\ref{dec_theorem}]
	
	\noindent
	{\it (2)}\quad This is a consequence the fact that $\intS{s,e}$ is indecomposable, for all $n$ and $s\leq e$ (see~\cite{general_krull_schmidt, auslander_reiten_smalo}).
	\smallskip
	
	\noindent
	{\it (1)}\quad
	Let $X$ in $\tame$ be cofibrant. 
	Choose a sequence $0=\tau_0<\cdots<\tau_k$ discretising $X$. 
	The morphism $X^{\tau_{a-1}<\tau_{a}}\colon X^{\tau_{a-1}}\to X^{\tau_{a}}$ is then a cofibration in $\mathbf{ch}$ for every $a=1,\ldots,k$.
	
	Assume first all the differentials in $X^{t}$ are trivial  for all $t$. 
	In this case, $X$ is isomorphic to $\bigoplus_{n\geq 0} X_{n}$.  
	Let $l_{n}^{0}:=\dim X_{n}^{0}$ and $l_{n}^{\tau_a}:=
	\dim\coker(
	X_{n}^{\tau_{a-1}<\tau_{a}}\colon X_{n}^{\tau_{a-1}}\to X_{n}^{\tau_{a}}
	)$ for $a=1,\ldots, k$. 
	Then $X_{n}$ is isomorphic to
	$\bigoplus_{a=0}^{k}\bigoplus_{j=1}^{l_{n}^{\tau_a}}\intS{\tau_{a},\infty}$
	and consequently $X$ is isomorphic to:
	\[
	\bigoplus_{n\geq 0}
	\bigoplus_{a=0}^{k}
	\bigoplus_{j=1}^{l_{n}^{\tau_a}}\intS{\tau_{a},\infty}
	\]
	
	Assume now there is a non-trivial differential in $X$ and set:
	\begin{enumerate}[(a)]
		\item $n$ to be the smallest natural number for which $\delta\colon X_{n+1}^t\to X_{n}^t$ is non trivial for some $t$. 
		This assumption implies $X_{n}^{t}=Z_{n}X^t$ for any $t$.
		\item $e$ to be the smallest $\tau_{a}$ for which $\delta\colon X_{n+1}^{\tau_a}\to X_{n}^{\tau_{a}}$ is non trivial.
		\item $s$ to be the smallest $\tau_{a}$ such that $\tau_{a}\leq e$ and for which the following intersection contains a non zero element:
		\[
		\Img\left(
		X_{n}^{\tau_{a}\leq e}
		\colon 
		Z_{n}X^{\tau_{a}}=X_{n}^{\tau_{a}}\hookrightarrow X_{n}^{e} 
		\right)
		\cap 
		\Img\left(\delta\colon X_{n+1}^{e}\to X_{n}^{e}
		\right)
		\not=0
		\]
	\end{enumerate}
	We claim that these choices imply $X$ is isomorphic to  $\intS{s,e}\oplus X'$.
	We can then apply the same strategy to $X'$.
	If $X'$ has a non-trivial differential, we split out of $X'$ another direct summand of the form $\intS[n']{s',e'}$ for $s'\leq e'$ in $\posr$. 
	Tameness guarantees that this process eventually terminates and we end up with an object with all the differentials being trivial, which we can decompose as described above and the theorem would be proven.
	
	It remains to show our claim that $X$ is isomorphic to  $\intS{s,e}\oplus X'$.
	For that, we make some choices:
	
	\begin{enumerate}
		\item Choose a non zero vector $v$ in the  intersection from step (c) above.
		\item Choose $x$ in $X_{n}^{s}=Z_{n}X^{s}$ and $y$ in $X_{n+1}^{e}$ such that:
		$
		X_{n}^{s\leq e}(x)=v=\delta(y)
		$.
		\item Consider the morphism $I(x,y)\colon \intS{s,e}\to X$ (see \ref{point_morphisms_out_interval_sphere_2}).
	\end{enumerate}
	The reason why we made all these choices is to ensure $I(x,y)\colon \intS{s,e}\to X$ is a cofibration (see Proposition~\ref{prop_cofibrations_from_interval_spheres}.2).
	
	Let $\phi\colon X\to \intS{s,s}$ be a morphism that fits into the following commutative diagram, where the top horizontal morphism is the standard inclusion (see \ref{point_morphisms_out_interval_sphere_2}). 
	Existence of such a $\phi$ is guaranteed by axiom \ref{MC4}:
	\[
	\begin{tikzcd}[row sep=15pt]
	\intS{s,e} \ar[d, hook', "{I(x,y)}"'] \ar[r]
	& \intS{s,s} \ar[d, two heads, "\sim"]
	\\
	X \ar[ur, dotted, "\phi"] \ar[r]
	& 0
	\end{tikzcd}
	\]
	If $t<e$, then the differential $\delta\colon X_{n+1}^{t}\to X_{n}^{t}$ is trivial.
	Thus, for any $s\leq t<e$, the linear function $\phi_{n+1}^{t}\colon X_{n+1}^{t}\to \intS{s,e}_{n+1}^{t}$ has to be trivial.
	It follows that $\phi\colon X\to \intS{s,s}$ factors through the standard inclusion:
	\[
	\begin{tikzcd}[column sep=40, row sep=13pt]
	& \intS{s,e}\ar{d}{i}
	\\
	X\ar{r}{\phi}\ar[bend left=10]{ru}{\psi} 
	& \intS{s,s}
	\end{tikzcd}
	\]
	The composition $\begin{tikzcd}
	\intS{s,e} \ar[hook]{r}{I(x,y)}
	& [0.2cm]X\ar{r}{\psi} 
	& \intS{s,e}
	\end{tikzcd}$ is therefore the identity and consequently $X$ is isomorphic to a direct sum $\intS{s,e}\oplus X'$.
\end{proof}

\section{Betti diagrams of objects}\label{section_Betti}
Let $X$ be a cofibrant object in $\tame$.
According to Theorem~\ref{dec_theorem}  there are unique  functions $\{\beta_nX\colon \Omega\to \{0,1,\ldots\}\}_{n=0,1,\ldots}$, called \textbf{Betti diagrams} of $X$, such that $X$ is isomorphic to:
\[
\bigoplus_n\bigoplus_{(s,e)\in\Omega} \left(\intS{s,e}\right)^{\beta_n X(s,e)}
\]
Betti diagrams have finite support:  the set $\text{supp}(\beta_nX):=\{(s,e)\in \Omega\ |\ \beta_nX(s,e)\not=0\}$ is finite for every $n$. 
Thus, to describe an isomorphism type of a  cofibrant object $X$ in $\tame$, a sequence of functions $\{\beta_nX\colon \Omega\to \{0,1,\ldots\}\}_{n=0,1,\ldots}$ with finite supports needs to be specified. 
Such functions are also called persistence diagrams (see \cite{edel_harer}).
Betti diagrams are complete invariants  of cofibrant objects in  $\tame$ and they play a  fundamental role in persistence and TDA.

In this section, we explain various ways of assigning Betti diagrams to arbitrary objects (not only cofibrant) in $\tame$.
For such general objects  one should not expect these invariants to be complete. 
Our strategy is to approximate arbitrary objects  by cofibrant objects and use Theorem~\ref{dec_theorem} to extract Betti diagrams from the obtained approximations. 

Minimal representatives (Definition~\ref{def_minimal_object}) and minimal covers (Definition~\ref{def_minimal_morphism}) are the most fundamental constructions that convert an arbitrary object in $\tame$ into a cofibrant one. 
This leads to two invariants of an isomorphism class of $X$ in $\tame$ which are called \textbf{minimal Betti diagrams} and \textbf{Betti diagrams}:
\[
\begin{tikzcd}[column sep =5pt]
& X\ar[bend right=10,mapsto]{dl}[swap]{\text{minimal  Betti diagrams}}\ar[bend left=10,mapsto]{dr}{\text{Betti diagrams}}\\
\{\beta_{n}X'\colon \Omega\to \{0,1,\ldots\}\}_{n=0,1,\ldots}
& & \{\beta_{n}\text{cov}(X)\colon \Omega\to \{0,1,\ldots\}\}_{n=0,1,\ldots}
\end{tikzcd}\]
where $X'$ is a minimal representative of $X$ and $\text{cov}(X)\to X$ is a minimal cover of $X$.
We also use the symbols $\beta_n^{\text{min}}X$ and $\beta_{n}X$ to denote respectively $\beta_{n}X'$ and $\beta_{n}\text{cov}(X)$. 
Moreover, if $X$ and $Y$ are weakly equivalent, then $\beta_n^{\min}X= \beta_n^{\min}Y$.

To understand relationship between these invariants, we first characterise minimal objects (see Definition~\ref{def_minimal_object}) in $\tame$.
Let $X$ be cofibrant in $\tame$. 
Consider its decomposition into a direct sum of interval spheres (Theorem~\ref{dec_theorem}). 
If this decomposition contains a component of the form $\intS{s,s}$, then by projecting it away, we obtain a self weak equivalence of $X$ which is not an isomorphism. 
Thus, if $X$ is minimal, its decomposition cannot contain such components, which is equivalent to having $\beta_nX(s,s)=0$ for all $n$ and all $s$. This implication can be reversed:

\begin{proposition}\thlabel{adfgsdfhgs}
	\begin{enumerate}
		\item 
		An object $X$ in $\tame$ is minimal if and only if it is cofibrant and  $\beta_nX(s,s)=0$ for all natural numbers $n$ and all $s$ in $[0,\infty)$.
		\item A morphism $c\colon X\to Y$ is a minimal cover if and only if $X$ is cofibrant, $c$ is a weak equivalence and a fibration, and
		no direct summand  of $X$ of the form $\intS{s,s}$, for some $n$ and $s$,  is in the kernel of $c$ (is mapped via $c$ to $0$).
	\end{enumerate}
\end{proposition}

\begin{proof}  
	We start with describing how a self weak equivalence of an object  $X$ in $\tame$ leads to its decomposition as a direct sum.
	Let $f\colon X\to X$ be a weak equivalence.
	Since  all objects   in  $\tame$ are compact, there is a natural number $l$ for which $f^l=f^{l+k}$ for all $k\geq 0$. We can then form the following commutative diagram where the top horizontal morphism is an isomorphism:
	\[\begin{tikzcd}[row sep=5pt]
	& \Img(f^l)\ar[bend left=10]{dr}\ar{rr} & & \Img(f^l)\ar[bend left=10]{dr} \\
	X\ar[two heads, bend left=10]{ur}\ar{rr}[description]{f^l} & & X\ar[two heads, bend left=10]{ur}\ar{rr}[description]{f^l}   & & X
	\end{tikzcd}\]
	Commutativity of this diagram, and the facts that the top horizontal morphism is an isomorphism and $f^l$ is a weak equivalence, have  two consequences: $X$ is isomorphic to a direct sum $\Img(f^l)\oplus \text{ker}(f^l)$, and the morphisms $X\twoheadrightarrow \Img(f^l)$
	and $\ker(f^l)\to 0$   are  weak equivalences. 
	
	\noindent
	{\it 1.}\quad  Assume $X$ is cofibrant and $\beta_nX(s,s)=0$ for all  $n$ and  $s$. 
	Let $f\colon X\to X$ be a self weak equivalence and $l$ be such that $X$ is isomorphic to $\Img(f^l)\oplus \text{ker}(f^l)$ and the morphism $\ker(f^l)\to 0$   is a  weak equivalence.
	According to Theorem~\ref{dec_theorem}, $\ker(f^l)$ is  a direct sum of interval spheres of the form $\intS{s,s}$. Since by the assumption $X$ does not have such components,  $\text{ker}(f^l)= 0$ and consequently $f^l$ and hence $f$ are isomorphisms.  That proves statement {\it 1}.
	
	\noindent
	{\it 2.}\quad 
	Consider a morphism $c\colon X\to Y$ such that $X$ is cofibrant, $c$ is a weak equivalence and a fibration.  If the kernel of $c$ contains a direct summand of $X$ of the form  $\intS{s,s}$, then by projecting it away, we would obtain
	a weak equivalence  $f\colon X\to X$ such that $cf=c$ and which is not an isomorphism, preventing $c$ to be a minimal cover.
	
	Assume now that the kernel of $c$ does not contain any direct summand of $X$ of the form $\intS{s,s}$. 
	Consider a weak equivalence $f\colon X\to X$ such that $cf=c$. 
	Choose $l$ for which $X$ is isomorphic to $\Img(f^l)\oplus \text{ker}(f^l)$ and the morphism $\ker(f^l)\to 0$ is a weak equivalence. 
	As before, Theorem~\ref{dec_theorem} ensures that $\ker(f^l)$ is a direct sum of interval spheres of the form $\intS{s,s}$. Since $\ker(f^l)$ is in $\text{ker}(c)$ and it is a direct summand of $X$, the assumption  implies $\ker(f^l)$ is trivial, and as before 
	$f$ is an isomorphism. 
\end{proof}

\begin{corollary}
	Let $X$ be an object in  $\tame$ (not necessarily cofibrant). 
	\begin{enumerate}
		\item A cofibrant object $X'$ in  $\tame$  is a minimal representative of $X$ if and only if it is weakly equivalent to $X$ and $\beta_nX'(s,s)=0$ for all natural numbers $n$ and all $s$ in $[0,\infty)$.
		\item Let $X'$ be the minimal representative of $X$ and $\text{\rm cov}(X)$ its minimal cover. 
		Then $\beta_{n}X'(s,e)=\beta_{n}\text{\rm cov}(X)(s,e)$ for all  $s<e$. 
	\end{enumerate}
\end{corollary}

According to the above corollary,  the minimal Betti diagrams and the Betti diagrams  may differ  only on the diagonal 
$\Delta:=\{(s,s)\in \Omega\}\subset \Omega$.  The minimal Betti diagrams ignore the diagonal by  assigning $0$ to all its  elements
(see Proposition~\ref{adfgsdfhgs}), reflecting the fact that minimal representative does not contain any component with trivial homology. 
The Betti diagrams on the other  hand  do not ignore  the diagonal and retain information about components of the cover $\text{cov}(X)$
that have trivial homology.

For certain objects in $\tame$, the minimal Betti diagrams and  the Betti diagrams  coincide and provide   a complete set of invariants:
\begin{proposition}\label{afdhggfbns}
	Assume $X$ and $Y$ in $\tame$ are such that all the differentials of  $X^t$ and $Y^t$ are trivial for all $t$ in $[0,\infty)$.
	Then:
	\begin{enumerate}
		\item $X$ and $Y$ are isomorphic if and only if they are  weakly equivalent. 
		\item $X$ and $Y$ are isomorphic if and only if their minimal Betti diagrams are equal.
		\item The minimal cover and the minimal representative of $X$ are isomorphic. 
		\item The minimal Betti diagrams and Betti diagrams of $X$ coincide.
	\end{enumerate}
\end{proposition}
\begin{proof}
	Statement {\it 1} is a consequence of the fact that $X$ and $Y$ are isomorphic to their respective homologies. 
	Statement {\it 2} follows from statement {\it 1}. 
	To show statement {\it 3}, choose  a minimal representative $X'$ of $X$ and a weak equivalence $f\colon X'\to X$. 
	Since $X$ is isomorphic to its homology, $f$ is a fibration and hence it is also a minimal cover of $X$. 
	Finally  statement {\it 4} is a consequence of  statement {\it 1}. 
\end{proof}

\begin{point}\label{asdfdfhdfn}
	{\em Tame $\posr$-parametrised vector spaces.}
	We regard {\bf tame $\posr$-parametrised vector spaces}, also known as persistence modules, as objects in $\tame$ whose values are concentrated only in degree $0$ for all parameters $t$ in $\posr$ (see \ref{point_graded_vect}). 
	Such objects in $\tame$ satisfy the assumption of Proposition~\ref{afdhggfbns} and hence their isomorphism types are uniquely determined by their Betti diagrams. 
	Furthermore, minimal representatives and minimal covers of such objects coincide. 
	It follows that for a tame 
	$\posr$-parametrised 
	vector space $X$, we have $\beta_nX(s,s)=0$ for all $n$ and $s$ (see Proposition~\ref{adfgsdfhgs}). Since homology in positive degrees of $X$ is trivial, we also get  
	$\beta_nX=0$ for all $n>0$. 
	Thus, the isomorphism type of $X$ is uniquely determined by its $0$-th Betti diagram $\beta_0X\colon \Omega \to \{0,1,\ldots\}$. 
	In this case $\beta_0X$ coincides with the usual persistence diagram of $X$ \cite{edel_harer}.
\end{point}

\section{Betti diagrams of morphisms}
In this section we explain various ways of   assigning  Betti diagrams to a morphism $g\colon X\to Y$ in $\tame$. 

\begin{point}{\em Minimal factorisations.}\label{adgssgfhjdghj}
	If $g\colon X\to Y$ is a cofibration, then the quotient $Y/g(X)$ is cofibrant and we can take its Betti diagrams $\beta_n\left(Y/g(X)\right)$.
	If $g\colon X\to Y$ is not a cofibration, we can consider   its minimal factorisation (see Definition~\ref{def_minimal_morphism}):
	\[\begin{tikzcd}[row sep=3pt]
	& A\ar[two heads, bend left=10]{dr}{\beta}[swap]{\sim}\\
	X\ar{rr}[description]{g} \ar[hook,bend left=10]{ur}{\alpha} & & Y
	\end{tikzcd}\]
	and assign to $g$ the Betti diagrams $\beta_n\left(A/\alpha(X)\right)$ of the quotient $A/\alpha(X)$.
	These Betti diagrams are invariants of the isomorphism type of $g$, and in the case $X=0$  recover  the Betti diagrams of $Y$ discussed in Section~\ref{section_Betti}.
\end{point}

\begin{point}\label{adfadfghsfgh}{\em The cover of the cofiber.}
	Instead of taking the minimal factorisation of $g$, we can  apply the cofiber construction (see \ref{point_cofiber}) parameterwise, to  obtain an exact sequence in $\tame$:
	\[\begin{tikzcd} 0\ar{r} &Y\ar{r}{i} & Cg\ar{r}{p} & SX\ar{r}& 0\end{tikzcd}\]
	We can then assign to $g$ the Betti diagrams $\beta_n\text{cov}(Cg)$ of the  minimal cover $\text{cov}(Cg)$ of the cofiber $Cg$.
	The  diagrams $\beta_n\text{cov}(Cg)$  depend on the isomorphism type of $g$, and as before, in the case $X=0$, recover the Betti diagrams of $Y$ discussed in Section~\ref{section_Betti}.
\end{point}

\begin{point}\label{adfgdfghs}{\em The cofiber of the covers.}
	We can extract a cofibrant object out of $g\colon X\to Y$ yet  in another way.
	Use axiom \ref{MC4} to choose a morphism $g'$ that fits into the following commutative square, where the vertical morphisms denote  the minimal covers:
	\[\begin{tikzcd}[row sep=13pt]
	\text{cov}(X)\ar[two heads]{d}[swap]{c_X}{\sim}\ar{r}{g'} & \text{cov}(Y)\ar[two heads]{d}{c_Y}[swap]{\sim}\\
	X\ar{r}{g} & Y
	\end{tikzcd}\]
	Since $\text{cov}(X)$ and $\text{cov}(Y)$ are cofibrant, then so is the cofiber $Cg'$, and  hence we can take its Betti diagrams
	$\beta_nCg'$. 
	Although in this construction we made a choice of $g'$, the obtained  Betti diagrams do not depend on it and hence provide   invariants of the isomorphism type of $g$.
	To prove this independence, consider  the morphism $C(c_X,c_Y)\colon Cg'\to Cg$, induced by the commutativity of the square above (see~\ref{point_cofiber}). 
	It fits into the following  commutative diagram with exact rows, where the indicated morphisms are weak equivalences, cofibrations, and fibrations:
	
	\[\begin{tikzcd}[row sep=13pt]
	0\ar{r} &  \text{cov}(Y)\ar[two heads]{d}[swap]{c_Y}{\sim}\ar[hook]{r}{i} & Cg'\ar[two heads]{d}{C(c_X,c_Y)}[swap]{\sim}\ar[two heads]{r}{p} & S\text{cov}(X)\ar[two heads]{d}{Sc_X}[swap]{\sim} \ar{r} & 0\\
	0\ar{r} & Y\ar{r}{i} & Cg \ar[two heads]{r}{p}  & SX \ar{r} & 0
	\end{tikzcd}\]
\end{point}
\begin{proposition}\label{adfadfghsd}
	The following  factorisation is minimal:
	\[\begin{tikzcd}[row sep=2pt]
	& Cg'\ar[two heads, bend left=10]{dr}{C(c_X,c_Y)}[swap]{\sim}\\
	\text{\rm cov}(Y)\ar{rr}[description]{i c_Y} \ar[hook,bend left=10]{ur}{i}& & Cg
	\end{tikzcd}\]
	Furthermore, if $X$ is cofibrant, then $C(c_X,c_Y)\colon Cg'\to Cg$ is a minimal cover.
\end{proposition}
Since the morphism $ic_Y\colon \text{cov}(Y)\to Cg$ does not depend on $g'$, neither does its minimal factorisation. 
Therefore, Proposition~\ref{adfadfghsd} implies that the isomorphism type of $Cg'$ does not depend on the choice of $g'$ and consequently neither $\beta_nCg'$.

\begin{proof}[Proof of Proposition~\ref{adfadfghsd}]
	Consider a self equivalence $f\colon Cg'\to Cg'$ of the factorisation. It fits into the following commutative diagram:
	\[
	\begin{tikzcd}[column sep=8, row sep=10]
	& & \text{cov}(Y)\ar[two heads,bend left=15]{dddl}[pos=0.6]{c_Y}[swap, pos=0.8]{\sim}\ar[hook]{rrr}[pos=0.2]{i}& & & Cg'\ar[two heads,bend left=15]{dddl}[pos=0.6]{C(c_X,c_Y)}[swap, pos=0.8]{\sim}\ar{dll}[swap]{f}\ar[two heads]{rrr}{p}
	& & & S\text{cov}(X)\ar[two heads,bend left=15]{dddl}{Sc_X}[swap]{\sim}\ar{dll}
	\\
	\text{cov}(Y)\ar[equal]{rru}
	\ar[two heads,bend right=10]{ddr}{\sim}[swap]{c_Y} 
	\ar[crossing over,hook]{rrr}[pos=0.3]{i}
	& & & Cg'\ar[two heads,bend right=10]{ddr}{\sim}[swap, pos=0.3]{C(c_X,c_Y)}
	\ar[crossing over,two heads]{rrr}[pos=0.3]{p}
	& & & S\text{cov}(X)
	\ar[two heads,bend right=10]{ddr}{\sim}[swap]{Sc_X}
	\\
	\\
	& Y \ar[hook]{rrr}[pos=0.2]{i}& & & Cg\ar[two heads]{rrr}{i} & & & SX
	\end{tikzcd}
	\]
	The induced morphism $S\text{cov}(X)\to S\text{cov}(X)$ is then a weak equivalence and hence an isomorphism as 
	$Sc_X\colon S\text{cov}(X)\to SX$ is the minimal cover.  The morphism $f$ is therefore an isomorphism as well.
	
	Assume $X$ is cofibrant. We are going to use the criteria from Proposition~\ref{adfgsdfhgs}
	to argue that in this case $C(c_X,c_Y)$ is a minimal cover.
	Since $\text{cov}(X)=X$, the following square is a pullback:
	\[\begin{tikzcd}[row sep=13pt]
	\text{cov}(Y)\ar[hook]{r}{i}\ar[two heads]{d}[swap]{c_Y}{\sim} & Cg'\ar[two heads]{d}{C(c_X,c_Y)}[swap]{\sim}\\
	Y\ar{r}{i} & Cg
	\end{tikzcd}\]
	It follows that the kernel of $C(c_X,c_Y)$ coincide with the kernel of $c_Y$. 
	Consider a direct summand  of $Cg'$ of the form   $\intS{s,s}$, for some $n$ and $s$, that belongs to the kernel of 
	$C(c_X,c_Y)$. 
	Then this summand belongs also to the kernel of $c_Y$. 
	In particular, it is included in $\text{cov}(Y)$ and hence it is
	also a direct summand of $\text{cov}(Y)$. 
	That contradicts the criteria from Proposition~\ref{adfgsdfhgs}  applied to $c_Y$.
	We can conclude that such summands do not exist and hence, according to the same criteria, $C(c_X,c_Y)$ is a minimal cover.
\end{proof}

With a morphism $g\colon X\to Y$ in $\tame$,  we have associated four cofibrant objects: $A/\alpha(X)$ (see~\ref{adgssgfhjdghj}), 
$\text{cov}(Cg)$ (see~\ref{adfadfghsfgh}), $Cg'$ (see~\ref{adfgdfghs}), and the minimal representative of $\text{cov}(Cg)$. 
These cofibrant objects lead to Betti diagrams  $\beta_n\left(A/\alpha(X)\right)$, $\beta_n\text{cov}(Cg)$, $\beta_nCg'$, and $\beta_n^{\text{min}}\text{cov}(Cg)$. Since all these cofibrant objects are weakly equivalent to each other, all these Betti diagrams agree for all  $(s,e)$ in $\Omega$ such that $s<e$.  
They may have different values only on the diagonal $\Delta\subset \Omega$. 

\begin{point}{\em Commutative ladders.}
	A {\bf commutative ladder} is by definition an object in $\tame$ whose values at all parameters are chain complexes which
	are non trivial only in degrees zero and one. 
	For example, $\intS[0]{s,s}$ is a commutative ladder. 
	Similarly, so is the Kan extension of
	$\D[1]\to 0$ with respect to a sequence $s<e$ of elements in $\posr$ (see \ref{point_kan_extension}).
	A tame $\posr$-parametrised vector space is also an example of a commutative ladder. 
	In general, however, in contrast to tame $\posr$-parametrised vector spaces,
	the minimal Betti diagrams of a commutative ladder can fail to be equal to its Betti diagrams. 
	For example, the minimal representative
	of $\intS[0]{s,s}$ is trivial, however its minimal cover is $\intS[0]{s,s}$. 
	Furthermore, again in contrast to tame vector spaces, Betti diagrams are not complete invariants of commutative ladders. 
	For example, $\intS[0]{s,s}$ and the Kan extension of
	$\D[1]\to 0$ with respect to a sequence $s<e$ of elements in $\posr$ have isomorphic minimal covers and hence same Betti diagrams. 
	
	If $f\colon X\to Y$ is a morphism of  tame $\posr$-parametrised vector spaces, then its cofiber
	$Cf$ is a commutative ladder. Any  commutative ladder $Z$ is the cofiber of its differential $\delta\colon  Z_1\to Z_0$,
	where we regard  $Z_1$ and $Z_0$ as $\posr$-parametrised vector spaces.  
	Thus, we can extract from $Z$ various Betti diagrams assigned to the morphism $\delta\colon  Z_1\to Z_0$.
	For example, we can consider the Betti diagrams of the quotient of the cofibration in the minimal factorisation of the differential $\delta\colon Z_1\to Z_0$ (see \ref{adgssgfhjdghj}). 
	We can also apply to $\delta$ the procedure described in \ref{adfgdfghs}
	to obtain another sequence of Betti diagrams. 
	Thus, a commutative ladder leads to four sequences of Betti diagrams. 
	These Betti diagrams are not arbitrary.  
	Let $\beta_n\colon\Omega\to\{0,1,\ldots\}$ be any of these Betti diagrams extracted from a commutative ladder.
	Since for $n\geq 1$, there are no non-trivial morphisms from the interval sphere $\intS{s,s}$ into any commutative ladder, $\beta_n(s,s)=0$ for $n\geq 1$ and $s$ in $[0,\infty)$. 
	As values of commutative ladders have no homology in degrees strictly bigger than $1$, we then also get $\beta_n=0$, for $n>1$.
\end{point}

\section{Zigzags}\label{section_zigzags}

\begin{point}{\em Discrete zigzags.}
	Throughout this section, $k$ is assumed to be a positive natural number.
	Elements of the set $\{r,l\}$ are called directions, $r$ stands for right and $l$ for left.
	Let  $c=(c_1,\ldots,c_k)$  be a sequence of directions i.e., elements  of  $\{r,l\}$. 
	Such a sequence determines a poset structure "$\to $" on $\{0,1,\ldots,k\}$ where, for $a<b$ in $\{0,1,\ldots,k\}$,  $a\leftarrow b$ if $c_{a+1}=\cdots=c_b=l$, and $a\to b$ if $c_{a+1}=\cdots=c_b=r$. 
	This poset is denoted by $[k]_c$ and the sequence $c$ is called its {\bf profile}. 
	A profile $c$ consisting  of only $r$'s is called standard and the induced poset structure on $\{0,1,\ldots,k\}$ is denoted by $[k]$. Here are graphical illustrations of   $[4]_c$ for 3 different profiles:
	\[0\leftarrow 1\leftarrow 2\rightarrow 3\rightarrow 4
	\ \ \ \ \ 
	0\rightarrow1\rightarrow 2\rightarrow 3\rightarrow 4
	\ \ \ \ \ 
	0\leftarrow1\rightarrow 2\leftarrow 3\rightarrow 4
	\]	
	
	To define a functor $X\colon [k]_c\to  \text{Ch}$, the following needs to be specified:
	\begin{itemize} 
		\item $k+1$ chain complexes $X^a$ for every $a$ in $\{0,1,\ldots,k\}$,
		\item  $k$ chain morphisms:  $X^{a-1\to {a}}\colon X^{a-1}\to X^{a}$ for every $a$ such that $c_a=r$, and $X^{a\to a-1} \colon X^{a}\to X^{a-1}$
		for every $a$ such that $c_a=l$.
	\end{itemize}
	
	A  \textbf{discrete zigzag} is by definition a functor of the form $X\colon [k]_c\to \text{Ch}$ for some $k$ and some profile $c$.
\end{point}

\begin{point}{\em Straightening zigzags.}
	Choose a profile $c=(c_1,\ldots,c_k)$. 
	For $a$ in $\{0,1,\ldots,k\}$, define its weight $w_a$ to be the size of the set $\{k\ |\  k\leq a\text{ and } c_k=l\}$.  The weight of $a$ is the number of $l$ directions in $c$ whose indexes are not bigger than $a$. 
	Thus, $w_0=0$, and $w_1=1$ if and only if $c_1=l$.
	
	In this paragraph we are going to explain how to convert a zigzag $X\colon [k]_c\to  \text{Ch}$, indexed by the poset $[k]_c$, into 
	a functor  $\overline{X}\colon [k]\to  \text{Ch}$ indexed by the standard poset $[k]$.
	For $a$ in $\{0,1,\ldots,k\}$ define:
	
	\[\overline{X}^{a}=\begin{cases}
	S^{w_a}X^{a} &\text{ if } a<k \text{ and } c_{a+1}=r\\
	S^{w_a}CX^{a+1\to a}& \text{ if }a<k \text{ and } c_{a+1}=l\\
	S^{w_k}X^k &\text{ if }a=k
	\end{cases}\]
	Thus, for $a<k$, the value $\overline{X}^{a}$ depends on the direction $c_{a+1}$. If $c_{a+1}=r$, then $\overline{X}^{a}$
	is the  $w_a$ suspension of $X^{a}$. If  $c_{a+1}=l$, then $\overline{X}^{a}$ is the $w_a$ suspension of the cofiber 
	$CX^{a+1\to a}$ (see \ref{point_cofiber}).
	
	Next we are going to define morphisms $\overline{X}^{a-1<a}\colon \overline{X}^{a-1}\to \overline{X}^{a}$ for every $a $ in $\{1,\ldots,k\}$.
	These morphisms depend on the directions $c_a$ and $c_{a+1}$ for $a<k$, and $c_k$ for $a=k$, and are defined as follows, using the morphisms  $i$ and $p$ as described  in \ref{point_cofiber}:
	\begin{itemize}
		\item Assume  either $a<k$,  $c_a=r$ and $c_{a+1}=r$, or  $a=k$ and $c_k=r$.  Then:
		\begin{itemize}
			\item  $w_{a}=w_{a-1}$,
			\item  $\overline{X}^{a-1}=S^{w_{a-1}}X^{a-1}=S^{w_{a}}X^{a-1}$,
			\item $ \overline{X}^{a}=S^{w_{a}}X^{a}$.
		\end{itemize} 
		The morphism  $\overline{X}^{a-1<a}\colon \overline{X}^{a-1}\to \overline{X}^{a}$ is set to be:
		\[
		\begin{tikzcd}[column sep=60pt, row sep=7pt]
		\overline{X}^{a-1}\ar[equal]{d}\ar[bend left=5]{r}{\overline{X}^{a-1<a}} &\overline{X}^{a}\ar[equal]{d} \\
		S^{w_{a}}X^{a-1}   \ar{r}{ S^{w_a}X^{a-1\to a}}  &S^{w_{a}}X^{a}
		\end{tikzcd}\]
		\item If $c_a=r$ and $c_{a+1}=l$, then:
		\begin{itemize}
			\item  $w_{a}=w_{a-1}$,
			\item $\overline{X}^{a-1}=S^{w_{a-1}}X^{a-1}=S^{w_{a}}X^{a-1}$,
			\item $\overline{X}^{a}=S^{w_{a}}CX^{a+1\to a}$.
		\end{itemize} 
		The morphism  $\overline{X}^{a-1<a}\colon \overline{X}^{a-1}\to \overline{X}^{a}$ is set to be the composition:
		\[\begin{tikzcd}[column sep=30pt, row sep=7pt]
		\overline{X}^{a-1}\ar[equal]{d}\ar[bend left=5]{rrr}{\overline{X}^{a-1<a}} & & &\overline{X}^{a}\ar[equal]{d} \\
		S^{w_{a}}X^{a-1}\ar{rr}{S^{w_a}X^{a-1\to a}} & &S^{w_{a}}X^{a}\ar[hook]{r}{S^{w_{a}}i}  & S^{w_{a}}CX^{a+1\to a}
		\end{tikzcd}\]
		\item Assume  either $a<k$,  $c_a=l$ and $c_{a+1}=r$, or  $a=k$ and $c_k=l$.  Then:
		\begin{itemize}
			\item  $w_{a}=w_{a-1}+1$,
			\item $\overline{X}^{a-1}=S^{w_{a-1}}CX^{a\to a-1}$,
			\item $\overline{X}^{a}=S^{w_{a}}X^{a}=S^{w_{a-1}}SX^{a}$.
		\end{itemize}
		The morphism  $\overline{X}^{a-1<a}\colon \overline{X}^{a-1}\to \overline{X}^{a}$ is set to be:
		\[\begin{tikzcd}[column sep=8pt, row sep=7pt]
		\overline{X}^{a-1}\ar[equal]{d}\ar[bend left=5]{rrrrr}{\overline{X}^{a-1<a}} &&& & &\overline{X}^{a}\ar[equal]{d} \\
		S^{w_{a-1}}CX^{a\to a-1} \ar[two heads]{rrrr}{S^{w_{a-1}}p} && &&S^{w_{a-1}}SX^{a}\ar[equal]{r} &S^{w_{a}}X^{a}
		\end{tikzcd}\]
		\item If $c_a=l$ and $c_{a+1}=l$, then:
		\begin{itemize}
			\item  $w_{a}=w_{a-1}+1$,
			\item  $\overline{X}^{a-1}=S^{w_{a-1}}CX^{a\to a-1}$,
			\item  $\overline{X}^{a}=S^{w_{a}}CX^{a+1\to a}=S^{w_{a-1}}SCX^{a+1\to a}$.
		\end{itemize}
		The morphism  $\overline{X}^{a-1<a}\colon \overline{X}^{a-1}\to \overline{X}^{a}$ is set to be the composition:
		\[\begin{tikzcd}[column sep=35pt, row sep=7pt]
		\overline{X}^{a-1}\ar[equal]{d}\ar[bend left=5]{rr}{\overline{X}^{a-1<a}} &&\overline{X}^{a}\ar[equal]{d} \\
		S^{w_{a-1}}CX^{a\to a-1} \ar[two heads]{r}{S^{w_{a-1}}p} & S^{w_{a-1}}SX^{a}=S^{w_{a}}X^{a}\ar[hook]{r}{S^{w_{a}}i} & S^{w_{a}}CX^{a+1\to a}
		\end{tikzcd}\]
	\end{itemize}
	
	For example, consider the following discrete zigzags of chain complexes:
	\[X=(\begin{tikzcd}
	X_0 &X_1\ar{l}[swap]{f_1} &X_2\ar{l}[swap]{f_2} \ar{r}{f_3} &X_3 \ar{r}{f_4} & X_4
	\end{tikzcd})\]
	\[Y=(\begin{tikzcd}
	Y_0\ar{r}{g_1}  &Y_1\ar{r}{g_2} &Y_2 &Y_3 \ar{l}[swap]{g_3} & Y_4\ar{l}[swap]{g_4}
	\end{tikzcd})\]
	Then:
	\[\overline{X}=\Big(\begin{tikzcd}[column sep=12pt, row sep =3pt]
	Cf_1\ar{rr}\ar[two heads,bend right=15]{dr}{p} & & SCf_2\ar[two heads]{rr}{p}&& S^2X_2 \ar{rr}{S^2f_3}&& S^2X_3 \ar{rr}{S^2f_4} &&  S^2X_4\\
	& SX_1\ar[hook,bend right=15]{ur}{Si}
	\end{tikzcd}\Big)\]
	\[\overline{Y}=\Big(\begin{tikzcd}[column sep=12pt, row sep =3pt]
	Y_0\ar{rr}{g_1}  && Y_1\ar{rr} \ar[bend right=15]{dr}{g_2} && Cg_3\ar{rr}\ar[two heads,bend right=15]{dr}{p} & & SCg_4 \ar[two heads]{rr}{Sp} & & S^2Y_4\\
	& & & Y_2\ar[hook,bend right=15]{ur}{i} & &SY_3\ar[hook,bend right=15]{ur}{Si} &
	\end{tikzcd}\Big)\]
\end{point}
\begin{point}{\em Morphisms between straightened zigzags.}
	Consider a natural transformation $f \colon X\to Y$  between two discrete zigzags $X,Y\colon [k]_c\to \Ch$. 
	It is a sequence of morphisms $f=\{f^a\colon X^a\to Y^a\}_{0\leq a\leq k}$ for which the following squares commute
	for all $a$ in $\{1,\ldots,k\}$:
	\[
	\begin{array}{c|c}
	\text{ if }c_a=r &  \text{ if } c_a=l\\
	\hline
	\begin{tikzcd}[column sep=45pt, row sep=13pt]
	X^{a-1}\ar{r}{X^{a-1\to a}} \ar{d}[swap]{f^{a-1}}& X^a \ar{d}{f^a}\\
	Y^{a-1}\ar{r}{Y^{a-1\to a}} &Y^{a}
	\end{tikzcd}
	&
	\begin{tikzcd}[column sep=45pt, row sep=13pt]
	X^{a-1}\ \ar{d}[swap]{f^{a-1}}& X^a \ar{d}{f^a} \ar{l}[swap]{X^{a\to a-1}}\\
	Y^{a-1}\ &Y^{a}\ar{l}[swap]{Y^{a\to a-1}}
	\end{tikzcd}
	\end{array}
	\]
	
	The following morphisms form a natural transformation denoted by $\overline{f}\colon\overline{X}\to \overline{Y}$:
	\[
	\overline{f}^a\colon \overline{X}^{a}\to \overline{Y}^a=
	\begin{cases}
	S^{w_a}f^a &\text{ if } a<m\text{ and } c_{a+1}=r
	\\
	S^{w_a}C(f^a)& \text{ if }a<m \text{ and } c_{a+1}=l
	\\
	S^{w_k}f^k  &\text{ if }a=k
	\end{cases}
	\]
	
	The association $f\mapsto \overline{f}$ defines an additive  functor. This functor is faithful i.e., it is injective on the set of morphisms. Furthermore, it commutes with direct sums, since taking suspensions and cofiber sequences commute with direct sums.
	In general, however, this functor fails to be full i.e., surjective on the set of morphisms. 
	To understand this failure we enumerate all  natural transformations of the form $\overline{X}\to \overline{Y}$, using the following algorithm.
	Choose an arbitrary natural transformations  $g\colon\overline{X}\to \overline{Y}$. 
	
	\begin{itemize}
		\item Let $g^k\colon S^{w_k}X^k\to S^{w_k}Y^k$ be the $k$-th component of $g$. 
		Define $\widehat{g}^k\colon X^k\to Y^k$ to be  $S^{-w_k}g^k$ (see \ref{point_suspension}).
	\end{itemize}
	For $1\leq a \leq k-1$,
	\begin{itemize}
		\item Assume $c_{a+1}=r$. 
		Let $g^{a}\colon S^{w_{a}}X^{a}\to S^{w_{a}}Y^{a}$ be the $a$-th component of $g$.
		Define $\widehat{g}^{a}\colon X^{a}\to Y^{a}$ to be  $S^{-w_{a}}g^{a}$.
		Since $g$ is a natural transformation, the following square commutes:
		\[
		\begin{tikzcd}[column sep=45pt, row sep=13pt]
		X^{a}\ar{r}{X^{a\to a+1}}\ar{d}[swap]{\widehat{g}^{a}} 
		& X^{a+1}\ar{d}{\widehat{g}^{a+1}} 
		\\
		Y^{a}\ar{r}{Y^{a\to a+1}} 
		& Y^{a+1}
		\end{tikzcd}
		\]
		\item Assume $c_{a+1}=l$. 
		Let 
		$g^{a}\colon S^{w_{a}}CX^{a+1\to a}\to S^{w_{a}}CY^{a+1\to a}$ 
		be the $a$-th component of $g$. 
		Recall the arguments in \ref{point_cofiber}.
		Since $g$ is a natural transformation, we have the following commutative diagram with exact rows:
		\[\begin{tikzcd}[row sep=13pt]
		0 \arrow{r} 
		& X^{a}\arrow[hook]{r}{i}\ar{d}
		& CX^{a+1\to a} \arrow[two heads]{r}{p} \ar{d}[swap]{S^{-w_{a}}g^{a}}
		& SX^{a+1} \arrow{r} \ar{d}{S\widehat{g}^{a+1}}
		& 0
		\\
		0\arrow{r} 
		& Y^{a}\arrow[hook]{r}{i} 
		& CY^{a+1\to a}\arrow[two heads]{r}{p} 
		& SY^{a+1}\arrow{r} 
		& 0
		\end{tikzcd}
		\]
		Define $\widehat{g}^{a}\colon X^{a}\to Y^{a}$ to be the left vertical morphism in this diagram. 
		This diagram leads to a homotopy commutative square with a choice of a homotopy:
		\[
		\begin{tikzcd}[column sep=45pt, row sep=13pt]
		X^{a}\ar{d}[swap]{\widehat{g}^{a}}
		& X^{a+1}\ar{l}[swap]{X^{a+1\to a}}\ar{d}{\widehat{g}^{a+1}}
		\ar[Rightarrow]{dl}[swap]{h_{a+1}}
		\\
		Y^{a} & Y^{a+1}\ar{l}{Y^{a+1\to a}}
		\end{tikzcd}
		\]
	\end{itemize}
	
	By applying this algorithm, we obtain a bijection between the set of natural transformations $\overline{X}\to \overline{Y}$ and the set of pairs consisting of a sequence of morphisms $\{\widehat{g}^a\colon X^a\to Y^a\}_{1\leq a\leq k}$ 
	and a sequence of homotopies $\{h^{a}\colon X^{a}\to Y^{a-1}\ |\ 1\leq a\leq k\text{ and }c_a=l\}$  such that, for all $a$ in $\{1,\ldots,k\}$:	
	\[
	\begin{array}{c|c}
	\text{ if }c_a=r & \text{ if } c_a=l
	\\
	\hline
	\begin{tikzcd}[column sep=45pt, row sep=13pt]
	X^{a-1}\ar{r}{X^{a-1\to a}} \ar{d}[swap]{\widehat{g}^{a-1}}
	& X^a \ar{d}{\widehat{g}^a}
	\\
	Y^{a-1}\ar{r}{Y^{a-1\to a}} 
	& Y^{a}
	\end{tikzcd}
	&
	\begin{tikzcd}[column sep=45pt, row sep=13pt]
	X^{a-1}\ \ar{d}[swap]{\widehat{g}^{a-1}} & X^a \ar{d}{\widehat{g}^a} \ar{l}[swap]{X^{a\to a-1}} \ar[Rightarrow]{dl}[swap]{h_a}\\
	Y^{a-1}\ &Y^{a}\ar{l}{Y^{a\to a-1}}
	\end{tikzcd}
	\end{array}
	\]
	
	Here is a consequence of this enumeration:
\end{point}

\begin{corollary}\label{indecomposable_zigzags}
	Let $X\colon[k]_c\to \text{\rm ch}$ be a discrete 
	zigzag of compact chain complexes. 
	Assume that $X^{a}$ has trivial differentials for all $a$ in $\{0,\ldots,k\}$. 
	Then $\overline{X}$ is indecomposable if and only if $X$ is indecomposable.
\end{corollary}

\begin{proof}
	Since the functor $X\mapsto\overline{X}$ commutes with the direct sum and is faithful, if $X$ is decomposable, then so is $\overline{X}$. 
	If $\overline{X}$ is decomposable, then there is an idempotent morphism $g\colon\overline{X}\to\overline{X}$ which is not an isomorphism.
	As the differentials of $X^{a}$'s are trivial, the morphisms $\{\widehat{g}_a\}_{1\leq a\leq k}$ form an idempotent natural transformation $X\to X$ (see \ref{point_cofiber}), which is not an isomorphism.
	Consequently, $X$ is also decomposable.
\end{proof}

\begin{definition}
	An object in $\tame$ is called a {\bf zigzag} if it is isomorphic to the Kan extension, along some sequence $\tau_0<\cdots<\tau_k$ in $\posr$, of a functor of the form $\overline{X}$ for some discrete zigzag $X\colon [k]_c\to \vect\subset\ch$ whose values are concentrated in degree $0$. 
	Such a zigzag in $\tame$ is called an {\bf incarnation} of $X$.
\end{definition}

We think about $\tame$ as an ambient category containing  various incarnations of discrete zigzags of the form $X\colon [k]_c\to   \vect\subset \ch$ indexed by posets $[k]_c$ for different $k$'s and different profiles $c$. 
Important properties of discrete zigzags are reflected well by their incarnations. 
For example, according to Corollary~\ref{indecomposable_zigzags}, a discrete zigzag $X\colon [k]_c\to\vect\subset\ch$ is indecomposable if and only if all (equivalently any) of its incarnations are indecomposable in $\tame$.  
We can also use the category $\tame$ and its morphisms to compare discrete zigzags indexed by different posets. 
Furthermore, the model structure on $\tame$ can be utilised to extract invariants of discrete zigzags through taking minimal representatives and minimal covers of their incarnations. 
However, the minimal cover is not a complete invariant for zigzags.
Consider the following non-isomorphic discrete zigzags:

\noindent
$\quad X:$
\begin{tikzcd}[ampersand replacement = \&, column sep=0.7cm]
K \arrow{r}{\left[\begin{smallmatrix} \id \\ 0 \end{smallmatrix}\right]}
\& K^{2} 
\& K^{2} \ar[l, "\id"'] \arrow{r}{\left[\begin{smallmatrix} \id & 0 \end{smallmatrix}\right]}
\& K,
\end{tikzcd}
$\qquad \quad Y:$
\begin{tikzcd}[ampersand replacement = \&, column sep=0.7cm]
K \arrow{r}{\left[\begin{smallmatrix} \id \\ 0 \end{smallmatrix}\right]}
\& K^{2} 
\& K^{2} \ar[l, "\id"'] \arrow{r}{\left[\begin{smallmatrix} 0 & \id \end{smallmatrix}\right]}
\& K
\end{tikzcd}

\noindent
The minimal covers of their incarnations $L\overline{X}$ and $L\overline{Y}$ along a sequence $\tau_0<\tau_1<\tau_2<\tau_3$ in $\posr$ coincide:
$\text{cov}(L\overline{X})\cong \text{cov}(L\overline{Y})\cong \intS[0]{\tau_0,\tau_1}^{2}\oplus \intS[1]{\tau_2,\tau_3}\oplus \intS[1]{\tau_2,\infty}$.
As a consequence, neither the minimal representative is a complete invariant for zigzags.

\paragraph*{Acknowledgments.}   
W.\ Chach\'olski was partially supported by VR and  the Wallenberg AI, Autonomous System and Software Program (WASP) funded by Knut and Alice Wallenberg Foundation.
B.\ Giunti was partially supported by INdAM-GNFM. This work was partially carried out by the last two authors within the activities of ARCES (University of Bologna), and by the second author during her PhD at the University of Pavia, with the partial support of INFN and the hospitality of the mathematics department of KTH Stockholm.

\bibliographystyle{plain}
	
\end{document}